\documentclass[12pt]{article}
\usepackage{amsmath}
\usepackage{amsfonts}
\usepackage{amssymb}
\usepackage[english]{babel}
\usepackage{graphicx}
\usepackage{amsthm}
\usepackage{hyperref}
\usepackage{marvosym,bbm,verbatim,stackengine}
\usepackage[vcentermath]{genyoungtabtikz}

\numberwithin{equation}{section}
\theoremstyle{plain}
\newtheorem*{theorem*}{Theorem}
\newtheorem*{lemma*}{Lemma}
\newtheorem{theorem}{Theorem}
\newtheorem{lemma}{Lemma}[section]
\newtheorem{corollary}[lemma]{Corollary}

\newtheorem{proposition}[lemma]{Proposition}
\newtheorem{innercustomthm}{Theorem}
\newenvironment{customthm}[1]
  {\renewcommand\theinnercustomthm{#1}\innercustomthm}
  {\endinnercustomthm}

\theoremstyle{definition}
\newtheorem{definition}[lemma]{Definition}
\newtheorem{remark}[lemma]{Remark}
\newtheorem{example}[lemma]{Example}

\def\l{\lambda}
\def\x{\tilde{x}}
\def\tr{\operatorname{tr}}
\def\T{{\bf T}}
\def\F{{^F}\!}
\def\K{{^K}\!}
\def\sym{\mathrm{sym}}

\usepackage{geometry}
\begin{document}
	
\title{A Limit Theorem for Stochastically Decaying Partitions at the Edge}
\author{In-Jee Jeong\textsuperscript{1} and Sasha Sodin\textsuperscript{2}}
\footnotetext[1]{Department of Mathematics, Princeton University,
Fine Hall, Washington Road, Princeton, NJ~08544-1000 USA.
E-mail: ijeong@math.princeton.edu.}
\footnotetext[2]{School of Mathematical
Sciences, Tel Aviv University, Tel Aviv, 69978, Israel and School of Mathematical 
Sciences, Queen Mary University of London, London E1~4NS, United Kingdom. E-mail:
sashas{\MVOne}@post.tau.ac.il. Supported in part by the European
Research Council start-up grant 639305 (SPECTRUM).}
\date{\today}

\maketitle


\section{Introduction}

In this paper, we study the asymptotic behavior of the first, second, and so on rows of stochastically decaying partitions. We establish that, with appropriate scaling in time and length, the sequence of rows converges to the Airy$_2$ line ensemble. 

This result was first established, in a more general setting, by Borodin and 
Olshanski \cite{BO}, who relied on the determinantal	 structure of the Poissonized correlation functions. Our argument is based on a different, combinatorial approach, developed by Okounkov \cite{Ok}. This approach may be useful in other problems in which no 
determinantal structure is available, and also highlights the similarity between
random partitions and random matrices.

\subsection{Partitions, Plancherel measures, and stochastic dynamical systems}

\subsubsection{Plancherel measures}

Let $G$ be a finite group, and let $\operatorname{Irrep}(G)$ be the set of
isomorphism classes of irreducible representations. Every class function
$f: G \to \mathbb{C}$ can be represented as a linear combination of 
the characters $\chi_\lambda$ corresponding to $\lambda \in \operatorname{Irrep}(G)$:
\[ f(g) = \sum_{\lambda \in \operatorname{Irrep}(G)} \widehat{f}(\lambda)
\chi_\lambda(g) \frac{\dim \lambda}{|G|}~, \quad
\widehat{f}(\lambda) = \sum_{g \in G} f(g) \frac{\overline{\chi_\lambda(g)}}{\dim \lambda}~.\]
The Plancherel measure $\mathbb{P}_G $ on $\operatorname{Irrep}(G)$ is defined by
$\mathbb{P}_G (\lambda) =\dim^2 \lambda /  |G|$. The name is justified by
the Plancherel equality
\[ \sum_{\lambda \in \operatorname{Irrep}(G)} |\widehat{f}(\lambda)|^2 \, \mathbb{P}_G(\lambda) = \sum_{g \in G} |f(g)|^2~.\]

\subsubsection{Partitions}

Let $S_n$ be the symmetric group. The irreducible representations of $S_n$
are indexed by partitions $\lambda$ of $n$ (denoted: $\lambda \vdash n$). These
are non-increasing sequences of non-negative integers $\lambda_1 \geq \lambda_2 \geq \lambda_3 \geq \cdots \geq 0$ such that $\sum \lambda_j = n$. We denote the Plancherel
measure on the symmetric group by $\mathbb{P}_n = \mathbb{P}_{S_n}$,  and the expectations under this measure by $\mathbb{E}_n$.

Given $\l \vdash n$, the size of $\lambda$ is $|\lambda| = n$, and the length of $\lambda$ (denoted $l(\lambda)$) is  the largest 
index $j$ such that $\lambda_j \geq 1$.  We visualize $\l$ as a Young diagram, i.e. the union of boxes with coordinates $(i,j)$, $1\le i \le l(\l)$ and $1 \le j \le \l_i$ (see Figure~\ref{fig:young}, where the $i$-axis is directed downwards and the $j$-axis is directed rightwards). The content of a box is defined by $\operatorname{ct}(\square = (i,j)) = j-i$. The conjugate partition $\lambda'$ is defined by
\[ \lambda'_j = \left| \left\{ i \, \mid \, \lambda_i \geq j \right\}\right|~;\]
the corresponding Young diagram is obtained by reflection about the $i=j$ axis.

\begin{figure}
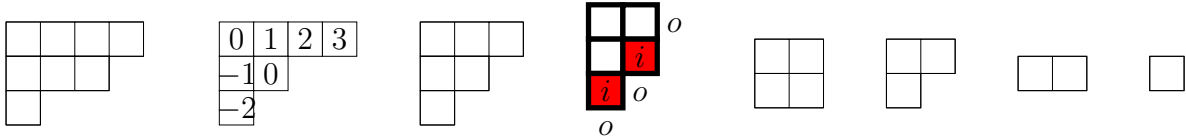

\newcommand\ylw{\Yfillcolour{yellow}}
\newcommand\wht{\Yfillcolour{white}}
\newcommand\red{\Yfillcolour{red}}
\[ \yng(4,3,1) \qquad 
\young(0123,<-1>0,<-2>)
\qquad \yng(3,2,1) \qquad 
\Ylinethick{2pt}
\gyoung(;;:o,;!\red<i>,;i:o,:o)
\Ylinethick{.3pt}
\qquad \yng(2,2)
\qquad \yng(2,1) \qquad \yng(2) \qquad \yng(1) \]
\caption{The Young diagrams corresponding to the  partitions 
$(4,3,1) \vdash 8$, $(4,2,1) \vdash 7$, $(3,2,1)\vdash 6$, $(2,2,1)\vdash5$, $(2,2)\vdash4$, $(2,1)\vdash3$, $(2)\vdash2$, $(1)\vdash1$. The numbers on the second diagram are the contents of its boxes. 
Each diagram is obtained from the previous one by removing a corner. The inner corners of the fourth diagram are marked with an {\em i} and colored in red, whereas the outer corners are marked with an~{\em o}.
}\label{fig:young}
\end{figure}

We say that $(i,\l_i)$ is an inner corner (or simply a corner) if $\l_i > \l_{i+1}$, and that
$(i, \l_i+1)$ is an outer corner if $\lambda_{i-1} > \lambda_i$ or $i = 1$. If we remove a corner box $\square_i = (i, \l_i)$ from $\l$, we get a partition of $n-1$ which we denote by $\l - \square_i$. 
If a partition $\mu \vdash n'$ is obtained from $\l \vdash n$ by consecutively removing some corners (or equivalently $\l$ is obtained from $\mu$ by consecutively adding outer corners) then we write $\mu \le \l$. Equivalently, $\mu \leq \lambda$ if $\mu_j \leq \lambda_j$ for any $j$.

\medskip

 The Frobenius coordinates of a partition $\lambda$ are the numbers $f_1 > \cdots > f_d \geq 0$ and $f_1' > \cdots > f_d' \geq 0$ defined by
\begin{equation}\label{eq:deffr} f_j = \lambda_j - j \quad (\lambda_j \geq j)~, \quad
f_j' = \lambda'_j - j \quad (\lambda'_j \geq j)~,\end{equation}
where $\lambda'$ is the conjugate partition.  The Kerov interlacing
coordinates \cite{Kerov} of a partition $\lambda$ are the numbers 
\[ \iota_1 > o_1 > \iota_2 > o_2 > \cdots > o_{c-1} > \iota_c~, \]
where $\iota_j$ and $o_j$ are the contents of the inner and outer corners of $\lambda$.

\begin{example} For the leftmost partition of Figure~\ref{fig:young},
\[ d = 2~, \, f_1 = 3~, \, f_2 = 1~, \, f_1' = 1~, \, f_2' = 0~, \]
and
\[ c = 4~,\, \iota_1 = 4~, \, o_1 = 3~, \, \iota_2 = 2~, \, o_2 =  1~, \, \iota_3 = -1, \,
o_3 = -2~, \, \iota_4 = -3~.\]
\end{example}

For more information about random partitions and the connections to other topics
such as the length of the longest increasing subsequence of a random permutation
we refer to the monograph of Romik \cite{Romik}.

\subsubsection{Plancherel decay}\label{sub:defdecay}  We now describe a stochastic system, as follows.
Let $n\geq 1$ be an integer. Define a sequence
of random partitions $(\Lambda^n(t) \vdash n - t)_{t = 0}^n$ as follows: $\Lambda^n(0) \vdash n$ is sampled from the Plancherel measure $\mathbb{P}_n$, and $\Lambda^n(t+1)$ is obtained
by erasing one of the corners from $\Lambda^n(t)$, so that 
\[ \mathbb{P}(\Lambda^n(t+1) = \lambda - \square_i \, \mid \, \Lambda^n(t) = \lambda) 
= \frac{\mathrm{dim}\,(\lambda - \square_i)}{\mathrm{dim}\,\l}~. \]
See Figure~\ref{fig:young} for a realization with $n = 8$.

The Plancherel measure is preserved by the dynamics. That is, for each $t \ge 0$, the distribution of $\Lambda^n(t)$ is given by $\mathbb{P}_{n - t}$; this is due to the  balance law
\begin{equation*}
\mathbb{P}_{n-t-1}(\mu) = \sum_{\nu \vdash n-t} \mathbb{P}_{n-t}(\nu) \mathbb{P}( \Lambda^n(t+1) = \mu  |  \Lambda^n(t) = \nu)
\end{equation*}
which is equivalent to the identity
\[ \dim \lambda = \sum_{\square_i} \dim(\lambda - \square_i)\]
expressing that the decomposition of $\lambda \vdash n-t$ into 
irreducible representations of $S_{n-t-1}$ contains no multiplicities.
 In particular, the distribution of $\Lambda^n(t)$ depends only on $n-t$. Note that this stochastic system is naturally time-reversible: given $n, t, t'$ with $n \ge t \ge t' \ge 0$, we have \begin{equation*}
\begin{split}
\mathbb{P}\left( \Lambda^n(t-t') = \mu | \Lambda^n(t) = \l  \right) = \mathbb{P}\left( \Lambda^{n+t'}(t+t') = \l | \Lambda^{n+t'}(t) = \mu \right)~,
\end{split}
\end{equation*} 
and the right hand side is well-defined for all $t' \ge 0$. Therefore, we may extend the (random) trajectory $\Lambda^n(t)$ to all $t < 0$; $\Lambda^n(t)$ has the distribution of $\mathbb{P}_{n+t}$. It follows from the induction rule that given any two partitions $\l \le \mu $ such that $\l \vdash n-t-1$ and $\mu \vdash n-t$, one has \begin{equation*}
\begin{split}
\mathbb{P}\left(   \Lambda^n(t-1) = \l    |  \Lambda^n(t) = \mu  \right) = \frac{\dim \mu}{(n-t)\dim \l}~.
\end{split}
\end{equation*}

\subsubsection{The Airy$_2$ line ensemble} The Airy$_2$ line ensemble 
is a stochastic process $(x_{j}(\tau))_{j \in \mathbb{Z}_{>0}, \tau \in \mathbb{R}}$ on $\mathbb{Z}_{> 0} \times \mathbb{R}$ such that
for any $\tau_1 < \cdots < \tau_k$ the collections of points $(x_{j}(\tau_1)), \cdots,
(x_{j}(\tau_k))$ form a determinantal process on $\mathbb{R}^k$ with kernel
\begin{equation}\label{eq:ak}
A(\tau'',u'';\tau',u') = \begin{cases}
\int_0^\infty e^{-u(\tau'' - \tau')} \mathrm{Ai}(u'' +u') \mathrm{Ai}(u'+u)\,du~, & \tau'' \ge \tau' \\
-\int_{-\infty}^0 e^{-u(\tau'' - \tau')} \mathrm{Ai}(u'' +u) \mathrm{Ai}(u'+u)\, du~, & \tau'' > \tau'
\end{cases}
\end{equation}
where $\mathrm{Ai}$ is the Airy function. We refer to  \cite{Bor_det} for a discussion
of determinantal  processes and to Figure~\ref{fig:airy} for an illustration.

The kernel (\ref{eq:ak})  first appeared in
the works of Mac\^edo \cite{Macedo} and Forrester, Nagao, and Honner \cite{FNH}.
The associated line ensemble was studied by Pr\"ahofer and Spohn \cite{PS},  Johansson \cite{Joh03}, Corwin and Hammond \cite{CH}; in particular, the existence of a
continuous modification was proved in these works. Numerous properties of the ensemble, particularly, invariance in distribution under a particular resampling, were proved in \cite{CH}, where also  the  term `Airy line ensemble'
was  coined. 

\begin{figure}
\includegraphics[trim={1.5cm 22.5cm  4cm 0},clip,scale=.5]{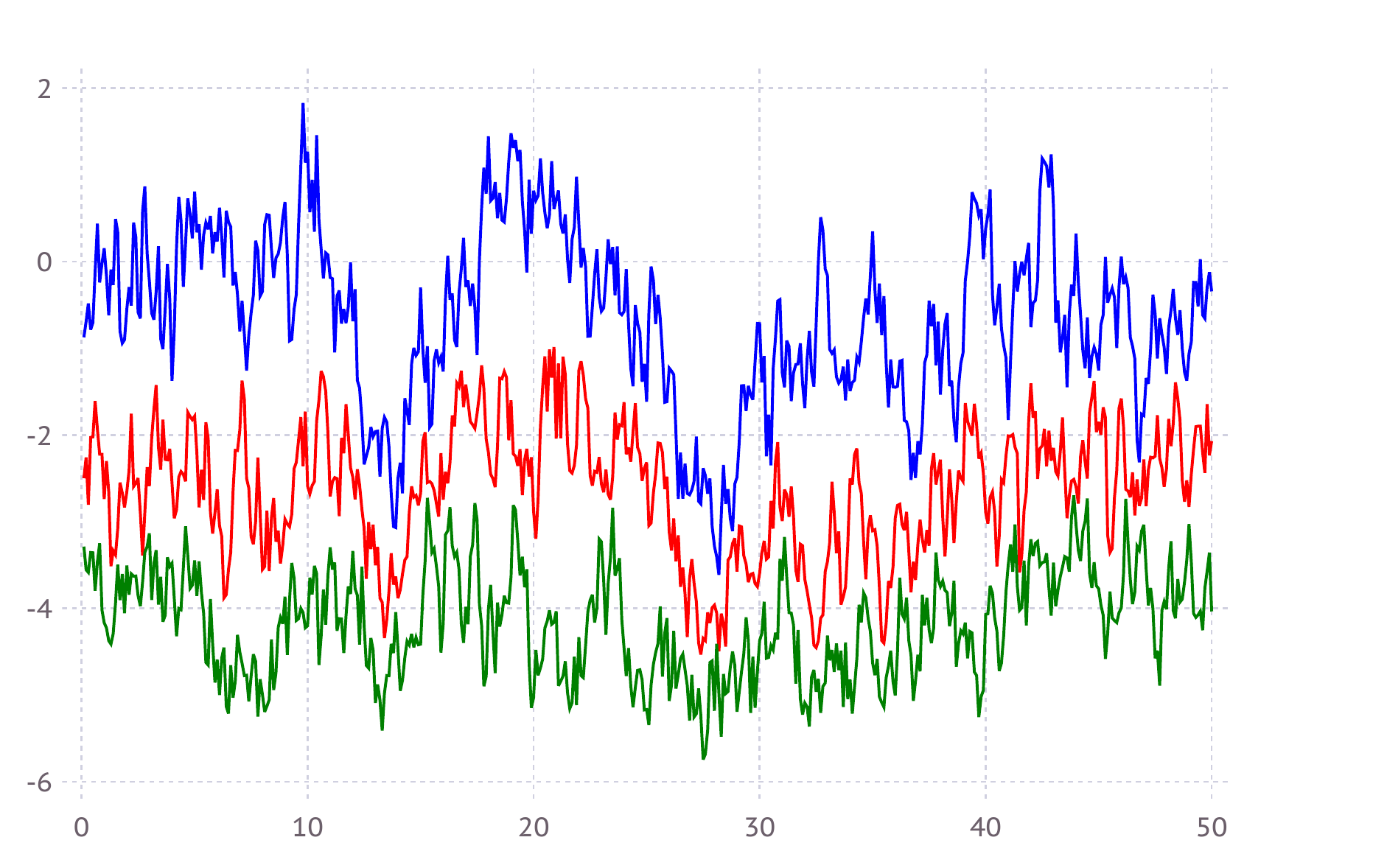}
	\centering
	\caption{The top three lines of the Airy$_2$ line ensemble (a realization). The top line fluctuates more than the next ones.}
\label{fig:airy}
\end{figure}

The joint distribution of the points $(x_j(\tau))$ for a fixed $\tau$ is described by
the Airy$_2$ point process of random matrix theory \cite{F,TW1}, 
which is a determinantal process with the Hermitian kernel
\begin{equation*}
A(u'';u') =
A(0,u'';0,u') =  \int_0^\infty \mathrm{Ai}(u'' +u) \mathrm{Ai}(u'+u)\,du~. \end{equation*}

\subsection{The main statement}

Consider the stochastic system $\Lambda^n(t)$ of \ref{sub:defdecay}. 
Introduce the rescaled variables 
\begin{equation}\label{eq:scaling}
x^n_j(\tau) = n^{-1/6} \left( \Lambda^n(t^n(\tau))_j - 2 (n - t^n(\tau))^{1/2} \right)~, \quad t^n(\tau) = 2\tau n^{5/6}~.
\end{equation} 
This defines $x^n_j(\tau)$ for $t^n(\tau) \in \mathbb{Z} \cap (-\infty,n]$. We then interpolate $x^n_j(\tau)$ as a piecewise linear function for $t^n(\tau) \le n$, and simply set $x^n_j(\tau) = x^n_j(n^{1/6}/2)$ for $\tau > n^{1/6}/2$. This gives for each $n \ge 1$, a random decreasing sequence of continuous functions \begin{equation*}
X^n(\tau) = ( x^n_1(\tau) \ge x^n_2(\tau) \ge \dotsb )~,\quad \tau \in \mathbb{R}~.
\end{equation*}

\begin{theorem}\label{thm:main}
	The sequence of  processes $X^n(\tau)$ converges to the Airy$_2$ line ensemble as $n \rightarrow \infty$, in the sense of finite-dimensional marginals.
\end{theorem}

The theorem is a special case of the (de-Poissonized) result of Borodin and Olshanski \cite{BO} ; their setting is described in \ref{sub:bo} below.

\smallskip

Specializing to the case $\tau = 0$, we recover the following result:

\begin{corollary}[Borodin-Okounkov-Olshanski \cite{BOO}, Johansson \cite{Jo}, Okounkov \cite{Ok}]\label{cor:BDJ}
	The sequence of point processes $\left( x_1^n \ge x_2^n \ge \dotsb \right)$ converges to the Airy$_2$ point process, in the sense of finite-dimensional marginals.
\end{corollary}

The convergence in distribution of $x_1^n$ and $x_2^n$ was first 
proved by Baik, Deift, and Johansson \cite{BDJ1,BDJ2}, who also conjectured
the full Corollary~\ref{cor:BDJ}.

\begin{remark} The scaling of the lengths of the partitions is natural in view of
the limit shape of $\lambda$, which was found by Logan--Shepp \cite{LS} and Vershik--Kerov \cite{VK1,VK2}, see further the book \cite{Romik} and also (\ref{eq:lsvk}) below.
\end{remark}

We also state a variant of Theorem~\ref{thm:main} in the Frobenius and Kerov coordinates. Let $f_j^n(t)$, and $f_j'^n(t)$ be the Frobenius coordinates of 
$\Lambda^n(t)$. Let $\iota_j^n(t)$ and $o_j^n(t)$ be the Kerov coordinates of $\Lambda^n(t)$, and let $\iota_j'^n(t)$ and $o_j'^n(t)$ be the Kerov coordinates of $\Lambda'^n(t)$.
Denote
\begin{equation}\label{eq:scalingF}
\begin{split}
\F{x}^n_j(\tau) &= n^{-1/6} \left(f_j^n(t(\tau)) - 2 (n - t(\tau))^{1/2} \right)~, \\
\F{x}_j'^n(\tau) &= n^{-1/6} \left(f_{j}'^n(t(\tau)) - 2 (n - t(\tau))^{1/2} \right)~, 
\end{split}
\end{equation} 
and 
\begin{equation}\label{eq:scalingK}
\begin{split}
\K{x}^n_j(\tau) &= n^{-1/6} \left(o_j^n(t(\tau)) - 2 (n - t(\tau))^{1/2} \right)~, \\
\K{x}_j'^n(\tau) &= n^{-1/6} \left(o_{j}'^n(t(\tau)) - 2 (n - t(\tau))^{1/2} \right)~.
\end{split}
\end{equation} 

\begin{customthm}{\ref{thm:main}$^\#$}\label{thm:main2}
	The  processes $\F{X}^n(\tau)$  and $\K{X}^n(\tau)$ converge to the Airy$_2$ line ensemble as $n \rightarrow \infty$, in the sense of finite-dimensional marginals.
\end{customthm}

\noindent Each of the statements of Theorem~\ref{thm:main2} is equivalent to 
Theorem~\ref{thm:main}: for the Frobenius version, the equivalence follows directly from the scaling,
while for the Kerov version, it is a consequence of the following fact (cf.\ \cite[Proposition 2]{Ok}): for any fixed $i,$
\[ \mathbb{P}_n\left\{ \text{$(i, \lambda_i)$ is a corner} \right\} \to 1~, \quad n \to \infty~. \]

\begin{remark}
	From the time-translation invariant nature of our stochastic system, it suffices to establish the result for $\tau \ge 0$ (which we assume for the rest of the paper). To extend it for $\tau \ge -1$ (say), our arguments can be repeated with $n$ replaced by $n' \approx n + 2n^{5/6}$; this simply shifts the parameter $\tau$ by 1. 
\end{remark}

\subsubsection{Strategy of the proof}
In a few lines, the strategy can be described as follows. 
Following the work of Okounkov \cite{Ok}, we consider expressions of the form 
\begin{equation}\label{eq:1}
\tr \left( \prod_{p=1}^k  X^{r_p}_{n - t_p- p+1} \right)
\end{equation}
with the exponents $r_p$ proportional to $n^{1/3}$ in the limit $n \rightarrow \infty$. Here, $X_1, X_2, \cdots$ are the Jucys--Murphy (JM) elements of the group algebra of $S_n$
(see (\ref{eq:jm}) below), and the trace is taken in the sense of the left regular representation (see Section~\ref{sub:GT}). The above quantity (after being suitably scaled) is an approximation of the mixed moments of the Laplace transforms of $X^n(\tau)$, which characterize the distribution. On the other hand, it counts the number of solutions to certain equations in $S_n$. 

To count these solutions, we combine the strategy of \cite{Ok} with the following 
construction, parallel to the one employed in \cite{FS} and in subsequent works on
random matrices which are surveyed in  \cite{Sodin2014}.  
For a suitably defined  family of polynomials $P_{l}^{n}$ of Chebyshev type  (see \ref{sub:op}), the modified moments
\begin{equation}\label{eq:2}
\tr \left( \prod_{p=1}^k P_{m_p}^{n-t_p-p}(X_{n-t_p - p+1}) \right)
\end{equation}
count the solutions to the same equations is $S_n$ which satisfy a certain 
irreducibility property.   We compute the asymptotics of (\ref{eq:2}) by classifying
them into topological equivalence classes.  The main
step in our argument is to show that the asymptotics of (\ref{eq:2}) matches the one appearing in \cite{Sodin2015}. Then we go back to the moments (\ref{eq:1}). 

\subsubsection{Plan of the paper}

The proof of the theorem exploits the similarity between the Frobenius coordinates
of a random partition and the eigenvalues of a random matrix. Therefore, in the next section, we briefly review several parallel results in random
matrix theory, particularly, the work of Soshnikov \cite{Sos}. Then in Section \ref{s:JM} we introduce the Jucys--Murphy elements, and reduce the theorems to the main technical estimate, Proposition~\ref{prop'}.

In Section \ref{s:mmm}, we describe the combinatorial constructions on which the proof is
based. Again, we emphasise the similarity to the constructions of random matrix theory.

The proof of Proposition~\ref{prop'} is deferred to Section \ref{S:main}, where it is preceded
by the proof of Lemma~\ref{lem:MMM} describing the asymptotics of modified moments
\eqref{eq:2}, and of Lemma~\ref{lem:MM} describing the asymptotics of \eqref{eq:1}.

\section{Parallel results in random matrix theory}\label{s:rm}

Theorem~\ref{thm:main} as well as Corollary~\ref{cor:BDJ} (the Baik--Deift--Johansson
conjecture proved in \cite{BOO,Ok,Jo}) bear a similarity to several results in the
theory of random matrices, which we now survey. 

\subsection{Wigner's law}
Consider an infinite Hermitian random matrix $H = (H(i,j) )_{i,j\ge 1}$  drawn from
the Gaussian Unitary Ensemble, meaning that $H(i, j) = (G(i, j) + \overline{G(j, i)})/\sqrt{2}$,
where $G(i, j)$ are  independent standard complex Gaussian entries. Take $N \ge 1$ and let us denote the eigenvalues of the principal submatrix $H^{(N)} =  (H(i,j))_{1 \le i,j \le N}$ by 
\[ \mu_1^{(N)} \ge \mu_2^{(N)} \ge \dotsb \ge \mu_N^{(N)}~.\]
According to Wigner's law \cite{Wi1,Wi2}  the sequence of appropriately scaled empirical measures converges to the (deterministic) semicircle measure; that is, 
\begin{equation}\label{eq:wig}
\frac{1}{N} \sum_{i=1}^N \delta_{\mu_i^{(N)}/N^{1/2}} \longrightarrow  \frac{1}{2\pi} \sqrt{(4-x^2)_+} dx
\end{equation}
as $N \rightarrow \infty$.  

Wigner's law is not specific for Gaussian entries, and remains valid for other
random matrix ensembles with independent entries. For example, it holds for
the ensemble of matrices $H$ in which $(H(i,j))_{i>j}$ are independent
and uniformly distributed on the unit circle, $H(j,i) = \overline{H(i,j)}$, and 
$H(i,i)=0$. This ensemble will play a r\^ole in the sequel (see e.g.\ (\ref{eq:cheb_rm})).

\subsection{Eigenvalues at the edge}\label{s:sosh}

Given Wigner's law, it is natural to consider the rescaled eigenvalues
\begin{equation}
y_i^{(N)} = N^{1/6} \left( \mu_i^{(N)} - 2N^{1/2} \right) \quad\mbox{and}\quad y_i'^{(N)} = -N^{1/6} \left( \mu_{N-i}^{(N)} - 2N^{1/2} \right)
\end{equation}
at each edge of the spectrum.

It turns out that the sequence of random measures $\sum_i \delta_{y_i^{(N)}}$ converges to the Airy$_2$ point process. This was first established for the Gaussian Unitary Ensemble by Tracy-Widom \cite{TW1} and Forrester \cite{F}, with the help of explicit formulas for the joint distribution of eigenvalues. 

Soshnikov \cite{Sos} proved that the Airy$_2$ limit is not specific to Gaussian matrices, 
and holds for arbitrary Hermitian matrices $H = (H(i,j))_{i,j\geq 1}$ provided that
the entries $(H(i,j))_{i\geq j\geq1}$ below the main diagonal are independent, and certain technical assumptions are satisfied. The assumptions of \cite{Sos} have been
since relaxed; the most general result was proved by Lee and Yin \cite{LeeYin}.

We briefly describe the strategy of Soshnikov.
Consider the  mixed moments, i.e.\ the expectations of products of  traces
\begin{equation}\label{eq:trace_RM}
\mathbb{E} \left( \prod_{p=1}^k \tr({H^{(N)}}^{r_p}) \right)~,\end{equation}
in the asymptotic regime in which $k$ is fixed and $r_p \sim 2\alpha_p N^{2/3}$.
We have
\begin{equation}\label{eq:trace_RM_expansion}
\begin{split}
\tr {H^{(N)}}^r &= 2^r N^{r/2} \left( \sum_{i=1}^N \left( \frac{\mu_i^{(N)}}{2N^{1/2}}\right)^r \right) \\
&\approx 2^r N^{r/2} \sum_{i \ge 1} \left( \exp(\alpha y_i^{(N)}) + (-1)^r \exp(\alpha y_i'^{(N)}) \right)~.
\end{split}
\end{equation}
Therefore for $ k \ge 1$,
\begin{equation}\label{eq:trace_expansion}
\begin{split}
 \prod_{p=1}^k \tr \left( \frac{H^{(N)}}{2N^{1/2}} \right)^{r_p}\approx  \prod_{p=1}^k \sum_{i \ge 1} \left( \exp(\alpha_p y_i^{(N)}) + (-1)^{r_p} \exp(\alpha_p y_i'^{(N)}) \right)~.
\end{split}
\end{equation}
More formally, the difference between the left-hand side and the right-hand side 
tends to zero conditionally on the event 
\[ \Omega_N = \left\{  \max(y_1^{(N)}, y_1'^{(N)}) \leq N^{1/10}\right\} \] 
(the constant $1/10$ has no special significance here). Let 
\[ \phi(\bar\alpha) = \phi(\alpha_1,\cdots,\alpha_k) = \mathbb{E} \prod_{p=1}^k \sum_i \exp(\alpha_p x_i) 
 \]
 be the Laplace transforms of the correlation functions of the  Airy$_2$ point process. Soshnikov proved the main
 estimate
 \begin{equation*}
 \begin{split}
\mathbb{E}\left( \prod_{p=1}^k \tr \left( \frac{H^{(N)}}{2N^{1/2}} \right)^{r_p} \right) - \sum_{I\subset \{1,\cdots,k\} } (-1)^{\sum_{i \in I} r_i} \phi(\overline{\alpha}|_I)\phi(\overline{\alpha}|_{I^c}) \longrightarrow 0
\end{split}
\end{equation*} 
in the asymptotic regime above. From this he deduced that $\mathbb{P}(\Omega_N) \to 1$ and that 
\begin{equation}\label{eq:sosh.conv}
\phi^{(N)}(\overline{\alpha}) = \mathbb{E} \prod_{p=1}^k \sum_i \exp(\alpha_p y_i^{(N)}) \mathbbm{1}_{\Omega_N} 
\longrightarrow \phi(\bar\alpha)~.
\end{equation}
This implies convergence of correlation functions, which in turn implies convergence
in distribution. We remark that (\ref{eq:sosh.conv}) also holds without restriction 
to $\Omega_N$, cf.\ the proof of Theorem~\ref{thm:main} in Section~\ref{s:pfthm} below.

\subsection{Stochastic setting}

The results discussed in the previous paragraph admit a generalization to a stochastic
setting, somewhat analogous to that of the current paper. In the case of the Gaussian
Unitary ensemble, the joint distribution of $\mu_j^{(N)}$ is determinantal, as proved
by Johansson--Nordenstam \cite{JohNor} and Okounkov--Reshetikhin \cite{OR}. The edge 
scaling limit was studied by Forrester--Nagao \cite{ForNag}, who proved the following. 
Consider the rescaled eigenvalues 
\begin{equation}\label{eq:scaling_rm}
y_j^{(N)}(\tau) = N^{1/6} \left(\mu_j^{(N+t(\tau))} - 2(N+t(\tau))^{1/2} \right)~,\quad t(\tau) = 2\tau N^{2/3}~.
\end{equation}  With a piecewise linear interpolation, we obtain a decreasing sequence of continuous functions $(y_1^{(N)}(\tau) \ge y_2^{(N)}(\tau) \ge \cdots)$. Forrester and Nagao showed that these converge to the Airy$_2$ line ensemble. Their argument makes use of the determinantal formul{\ae}.

The Forrester--Nagao result was re-proved and generalized in \cite{Sodin2015} using
the method of (modified) moments. A byproduct of the argument described there
is a description of the Laplace transform of the Airy$_2$ line ensemble,
\begin{equation}\label{eq:phi1} \phi(\bar\alpha,\bar\tau) =  \mathbb{E} \prod_{p=1}^k \sum_i \exp(\alpha_p x_i(\tau_p))~. 
\end{equation}
The description is reproduced in Section~\ref{s:cont} below; this, 
rather than (\ref{eq:ak}), is the description that we use in the proof of the main theorem.

\section{Spectra of JM elements and mixed moments}\label{s:JM}

\subsection{Gelfand--Tsetlin decomposition}\label{sub:GT}

The left regular representation of $S_n$ is the vector space of formal linear combinations
$\sum_{\pi \in S_n} c_\pi \pi$ with the action of $S_n$ by left multiplication. It has a decomposition
 $\oplus_{\lambda \vdash n} \lambda^{\mathrm{dim}(\lambda)}$, where we slightly abuse
 notation and denote the representation space corresponding to $\lambda$
 by the same letter $\lambda$. 
 
 The Gelfand-Tsetlin decomposition is a decomposition of each space $\lambda$ into a direct sum of one-dimensional subspaces, which is constructed as follows (cf.\ \cite{CST}). Consider the collection of chains (semistandard Young tableaux)
 \[ \mathrm{Tab}(\lambda)= \left\{ \T = (\lambda^0 = \varnothing \leq \lambda^1 \leq \lambda^2 \leq \cdots \leq \lambda^n = \lambda)~, \quad |\lambda_k| = k\right\} ~. \]
 To each tableau $\T$, we associate the space $\T$ of vectors $v \in \lambda$ such that,
 for every $j$, $v$ lies in the representation of $S_j$ which is isomorphic to $\lambda^j$. 
 Due to lack of multiplicities (cf.\ \ref{sub:defdecay}),  $\dim \T = 1$, hence $\lambda=\oplus_{\T \in \mathrm{Tab}(\lambda)} \T$. Thus $\T \mapsto \dim \T / n!$ is a probability
 distribution on tableaux of size $n$. This distribution coincides with the stochastic system defined in \ref{sub:defdecay}; in particular, its projection to $\lambda^n$ coincides with $\mathbb{P}_n$.
 
The analogy with the random matrix setting of Section~\ref{s:rm} is as follows. Pick a random tableau $\T$ according to the distribution specified above. The Frobenius 
coordinates $\{f_j^n\} \cup \{-f_j'^n\}$ are the counterpart of the eigenvalues 
$\{\mu_j^{(N)}\}$ of $H^N$. Note that the number of Frobenius coordinates
is not constant, however, it is strongly concentrated about $\frac{4}{\pi}\sqrt{n}.$ Therefore we think of the Frobenius coordinates of a random partition as the counterpart
of the eigenvalues of a random matrix of dimension $\frac4\pi \sqrt{n}$. 

The counterpart of Wigner's law (\ref{eq:wig}) is  the limit shape of \cite{LS,VK1,VK2}: in terms of the rescaled Frobenius coordinates, we have
\begin{equation}\label{eq:lsvk}
\frac{\pi}{4\sqrt{n}} \sum_j \left\{ \delta_{f_j^n/\sqrt{n}} + \delta_{f_j'^n/\sqrt{n}} \right\} \longrightarrow  \frac{1}{4} \arccos\left( \frac{|x|}{2} \right) \cdot \mathbbm{1}_{[-2,2]} dx~, \end{equation}
see Figure~\ref{fig:lskv}.

\begin{figure}
	\includegraphics[scale=.4]{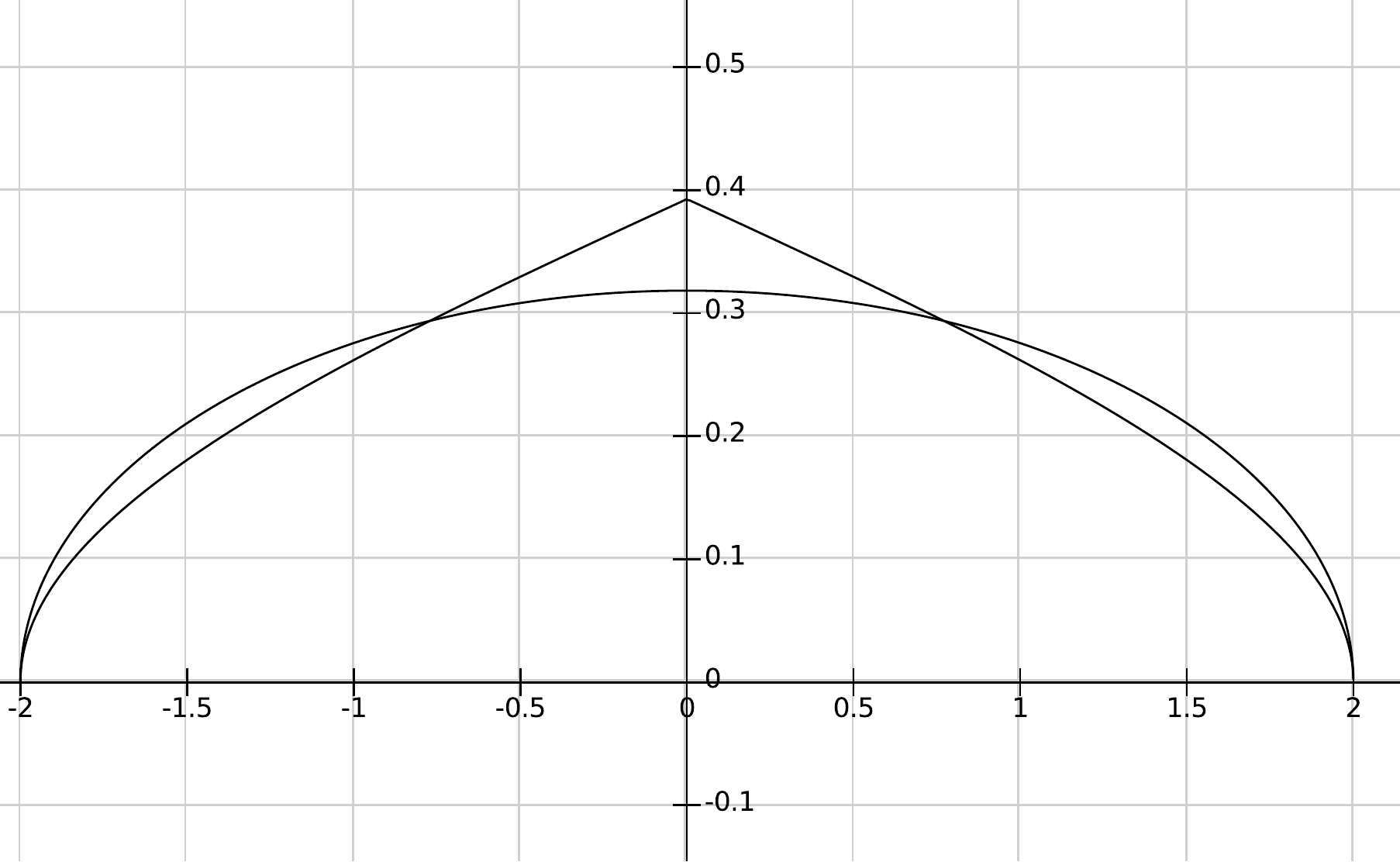}
	\centering
	\caption{The Logan--Shepp--Vershik--Kerov limit shape in Frobenius coordinates and Wigner's semicircle}\label{fig:lskv} 
\end{figure}

We base the
discussion below on this simple-minded analogy. A more precise analogy between random partitions and random matrices, in which
the counterpart of the eigenvalue distribution of a random matrix
is the so-called transition measure
of a partition, was developed by Kerov, see \cite{Kerov} and the monograph \cite{Ker_book}.

\subsection{JM elements and their spectra}\label{sub:JM}

The argument of Okounkov \cite{Ok} which we develop in this paper is based on
the interpretation of the Frobenius coordinates of a partition as eigenvalues
of certain elements of the group algebra of $S_n$, considered as operators via
the left regular representation. First consider the JM elements, which are elements in the group algebra of $S_n$ defined as the following sums of transpositions: 
\begin{equation}\label{eq:jm}
\begin{split}
X_1 &= 0 \\
X_2 &= (1 2) \\
X_3 &= (13) + (2 3) \\
\cdots & \\
X_n &= (1 n) + (2 n) + \cdots + (n-1 \, n)
\end{split}
\end{equation}
These elements were introduced by  A.\ A.\ Jucys and G.\ E.\ Murphy, after whom
they are now named.\footnote{Following Stanley \cite{St}, we mention that these 
elements  appeared in a paper of A.\ Young \cite{Young} published in 1901.} 
 In the work of Okounkov and Vershik \cite{OV}, see further \cite{CST}, the JM elements
elements are used as a starting point to reconstruct the representation theory
of the symmetric group. Here we recall some basic facts about the spectra of JM elements, following \cite{Ok} and \cite{OV}. 
 
The subspaces $\T$ are invariant for all the Jucys--Murphy elements. Namely,
\[ X_n|_{\T} = \operatorname{ct}(\square)~,\quad \text{where} \quad \lambda^n = \lambda^{n-1} + \square~. \]
Therefore the restriction of $X_n$ to $\lambda$ has eigenvalues $\operatorname{ct}(\square_i) = \lambda_i - i$ with multiplicity $\mathrm{dim}(\lambda - \square_i)$ where $\square_i = (i,\lambda_i)$ ranges over the corners of $\lambda$. Now the  spectra of $X_{n-1}, X_{n-2}$, and so on can also be described, since $X_{n-k+1}$ plays the role of the first JM element in the subgroup $S_{n-k+1}$ (viewed as the collection of permutations fixing the letters $n-k+2,\cdots,n$):
\[ X_{n-k+1}|_ {\T} = \operatorname{ct}( \lambda^{n-k+1} /\lambda^{n-k}) =\operatorname{ct}(\square)~, \quad \text{where} \quad \lambda^{n-k+1} = \lambda^{n-k} + \square~.\]
That is, each eigenspace of $X_n|_\lambda$ which has dimension $\dim(\l - \square_i)$ splits into eigenspaces of $X_{n-1}$ with dimensions $\dim((\l - \square_i) -\square_j)$ where $\square_j$ ranges over the corners of $\l - \square_i$, and so on.

\begin{example} The sequence of Young diagrams on Figure~\ref{fig:young}
forms a Young tableau $\T$ of size $8$. In this case,
\[\begin{split}
&X_1|_\T = 0~, \, X_2|_\T =  1~, \, X_3|_\T = -1, \, X_4|_\T = 0~, \\
&X_5|_\T = -2~, \, X_6|_\T = 2~,\,  X_7|_\T = 3~, \, X_8|_\T = 1~.
\end{split}\]
\end{example}

\subsection{Powers of JM elements}
According to \ref{sub:JM}, the $r$-th power of a JM element acts on a tableau
$\T$ by
\[ X_{n-k+1}^r|_ {\T} = \operatorname{ct}( \lambda^{n-k+1} /\lambda^{n-k})^r~.\]
Introduce the elements 
\begin{equation}\label{eq:Y}
\begin{split}
Y_{m,r} = \sum_{q=1}^m X_q^r~, \quad 1 \leq m \leq n~, r \geq 0~.
\end{split}
\end{equation}
Then 
\begin{equation}Y_{m,r}|_{\T} 
= \sum_{\square \in \lambda^m} \operatorname{ct}(\square)^r 
= \sum_j \sum_{i=j+1}^{\lambda^m_j} (i-j)^r
+(-1)^r \sum_j \sum_{i=j+1}^{\lambda'^m_j} (i-j)^r~.
\end{equation}
Introducing the notation
\[ S_p(l) = \sum_{i=1}^l i^p \]
for power sums and recalling the definition (\ref{eq:deffr}) 
of Frobenius coordinates, we obtain:
\begin{equation}\label{eq:Yact} 
Y_{m,r}|_{\T} 
= \sum_j S_r(f_j^m) + (-1)^r \sum_j S_r(f'^m_j)~. 
\end{equation}
Note the similarity to the random matrix moments
\[ \tr (H^{(N)})^r = \sum_{\mu_j \geq 0} (\mu_j^{(N)})^r + (-1)^r \sum_{\mu_j < 0} (-\mu_j^{(N)})^r~.  \]

\begin{remark}  The elements $Y_{m,r}$ are symmetric expressions (\ref{eq:Y}) in the JM elements $X_1, \cdots, X_n$, and are also functions (\ref{eq:Yact}) of the Frobenius coordinates. This is consistent with a theorem of Jucys \cite{Jucys} according to which 
the algebra of symmetric functions of $X_1, \cdots, X_n$
coincides with the center of the group algebra of $S_n$. 
\end{remark}

\subsection{Mixed moments and the proof of the main theorem}\label{s:pfthm}
Consider the (deterministic) expressions 
\begin{equation}\label{eq:JMsymtraces}
\mathcal{M}^\sym(\overline{r},\overline{t})  = \frac{1}{n!} \tr \prod_{p=1}^k \left( \frac{r_p}{(4n_p)^{\frac{r_p+1}2}} Y_{n_p, r_p} \right)\quad\mbox{with}\quad n_p = n-t_p~.
\end{equation}
According to (\ref{eq:Yact}), 
\begin{equation}\label{eq:JMsymtraces'}
\mathcal{M}^\sym(\overline{r},\overline{t}) = \mathbb{E}
\prod_{p=1}^k \left( \sum_j \frac{r_p}{(4n_p)^{\frac{r_p+1}2}} S_{r_p}(f^{n_p}_j) + (-1)^{r_p}\sum_j \frac{r_p}{(4n_p)^{\frac{r_p+1}2}} S_{r_p}(f'^{n_p}_j)\right)~.
\end{equation}

\begin{proposition}\label{prop'} We have a bound
\begin{equation}\label{eq:JM_bound}
\begin{split}
\mathcal{M}^\sym(\overline{r},\overline{t}) \le  \prod_{p=1}^k \frac{Cn^{1/2}}{r_p^{3/2}} \exp(C_k r_p^3/n)~.
\end{split}
\end{equation}
	 Moreover, in the asymptotic regime 
\begin{equation}\label{eq:JM_asymp}
r_p \sim  2\alpha_p n^{1/3}~,\qquad t_p = 2\tau_p n^{5/6}~,
\end{equation}
we have:
\begin{equation}\label{eq:convergence}
\mathcal{M}^\sym(\overline{r},\overline{t}) - 
\sum_{I\subset \{1,\cdots,k\} } (-1)^{\sum_{p \in I} r_p} \phi(\overline{\alpha}|_I, \bar\tau|_I)\phi(\overline{\alpha}|_{I^c}, \bar\tau_{I^c}) \longrightarrow 0~.
\end{equation}
\end{proposition}
The proposition will be proved in Section~\ref{S:main}.

\begin{proof}[Proof of Theorem~\ref{thm:main}]
We prove the Frobenius version of Theorem~\ref{thm:main2}. Let us show  that 
for any $k \geq 1$ and any $\tau_1, \cdots, \tau_k \geq 0$, $r_1, \cdots, r_p = \pm 1$
\[\begin{split} 
&\mathbb{E} \prod_{p=1}^k \left\{ \sum_j \exp(\alpha_p \F{x}_j^n(\tau_p)) + (-1)^{r_p} \sum_j\exp(\alpha_p \F{x}_j'^n(\tau_p)) \right\} 
\\
&\qquad\to \sum_{I\subset \{1,\cdots,k\} } (-1)^{\sum_{p \in I} r_p} \phi(\overline{\alpha}|_I, \bar\tau|_I)\phi(\overline{\alpha}|_{I^c}, \bar\tau_{I^c})~, \quad n \to \infty~.
\end{split}\]
This will imply that the Laplace transforms of the correlation measures of 
$\F{X}^n(\tau)$ converge to those of the Airy$_2$ line ensemble, and therefore
the processes converge in distribution.

Denote 
\[ \Omega_{n} = \left\{ \max_{1 \leq p \leq k} \max(\F{x}^n_1(\tau_p), \F{x}'^n_1(\tau_p)) < n^{1/10}\right\}~. \]
It suffices to show that
\begin{equation}\label{eq:bd.goodevent}\begin{split} 
&\mathbb{E} \mathbbm{1}_{\Omega_n} \prod_{p=1}^k \left\{ \sum_j \exp(\alpha_p \F{x}_j^n(\tau_p)) + (-1)^{r_p} \sum_j\exp(\alpha_p \F{x}_j'^n(\tau_p)) \right\} 
\\
&\qquad\to \sum_{I\subset \{1,\cdots,k\} } (-1)^{\sum_{p \in I} r_p} \phi(\overline{\alpha}|_I, \bar\tau|_I)\phi(\overline{\alpha}|_{I^c}, \bar\tau_{I^c})~, \quad n \to \infty~.
\end{split}\end{equation}
and that 
\begin{equation}\label{eq:bd.badevent}
\mathbb{E} \mathbbm{1}_{\Omega_n^c} \prod_{p=1}^k \left\{ \sum_j \exp(\alpha_p \F{x}_j^n(\tau_p)) + (-1)^{r_p} \sum_j\exp(\alpha_p \F{x}_j'^n(\tau_p)) \right\} \to 0~.\end{equation}
Introduce the events 
\begin{equation*}
\begin{split}
B_{l,n} &= \left\{ 2^l \le \max\left(\frac{f_1^{n}}{2n^{1/2}}-1,\frac{f_1'^{n}}{2n^{1/2}}-1\right) \le 2^{l+1}   \right\}~, \quad l \geq l_0 = - \lceil \frac{7}{30} \log_2 n \rceil - 2~.
\end{split}
\end{equation*} 
By the case $k=1$ of Proposition~\ref{prop'} and the Chebyshev inequality,
\[ \mathbb{P}(B_{l,n}) < \exp(-cn^{1/2}2^{3l/2})~, \quad l \geq l_0~.\] 
The complement of $\Omega_n$ is contained in the union of $B_{l,n_p}$ over $l \geq l_0$ and $1 \leq p \leq k$, therefore  $\mathbb{P}(\Omega_n) \to 1$. Moreover, on the set $B_{l,n_p}$ we have \begin{equation*}
\begin{split}
\sum_j \exp(\alpha_1 \F x_j^n(\tau_p))+ \sum_j \exp(\alpha_1 \F x_j'^n(\tau_p)) < n \exp\left( Cn^{1/3}2^l \right)~,
\end{split}
\end{equation*} 
hence (\ref{eq:bd.badevent}) holds.

The proof of (\ref{eq:bd.goodevent}) follows the strategy of \cite{Sos} 
as described in \ref{s:sosh}. First,  the bounds above imply that 
\begin{equation}
 \mathbb{E} \mathbbm{1}_{\Omega_n^c} \prod_{p=1}^k \left( \frac{r_p}{(4n_p)^{\frac{r_p+1}2}} \left[ \sum_j S_{n_p}(f_j^{n_p})+\sum_j S_{n_p}(f_j'^{n_p})\right] \right) \to 0~.
 \end{equation}
 Therefore by Proposition~\ref{prop'}
 \begin{equation}\label{eq:cond1}
 \begin{split}
 &\mathbb{E} \mathbbm{1}_{\Omega_n} \prod_{p=1}^k \left( \frac{r_p}{(4n_p)^{\frac{r_p+1}2}} \left[ \sum_j S_{n_p}(f_j^{n_p})+\sum_j S_{n_p}(f_j'^{n_p})\right] \right) \\
 &\qquad\to \sum_{I\subset \{1,\cdots,k\} } (-1)^{\sum_{p \in I} r_p} \phi(\overline{\alpha}|_I, \bar\tau|_I)\phi(\overline{\alpha}|_{I^c}, \bar\tau_{I^c})~.
 \end{split}
 \end{equation}
Next, the contribution of $j$ with
\[ \frac{f_j^{n_p}}{2\sqrt{n_p}} < 1 - n_p^{-1/3+1/10} 
\quad \text{or} \quad \frac{f_j'^{n_p}}{2\sqrt{n_p}} < 1 - n_p^{-1/3+1/10}  \]
to (\ref{eq:cond1}) tends to zero, therefore (\ref{eq:cond1}) remains valid if the
$p$-th sum is restricted to $j$ such that 
\[ \left| \frac{f_j^{n_p}}{2\sqrt{n_p}} - 1\right| < n_p^{-7/30}~, \quad
 \left| \frac{f_j'^{n_p}}{2\sqrt{n_p}} - 1\right| < n_p^{-7/30}~.\]
For such $j$,
\[ \left|\frac{r_p}{(4n_p)^{\frac{r_p+1}2}} S_{n_p}(f_j^{n_p}) - \exp(\alpha_p \F{x}_j^n(\tau_p))\right| \leq C n_p^{-2/15} \frac{r_p}{(4n_p)^{\frac{r_p+1}2}} S_{n_p}(f_j^{n_p})  \]
and we obtain (\ref{eq:bd.goodevent}).
\end{proof}

\section{A modification of the moment method}\label{s:mmm}

\subsection{Modified mixed moments}

\subsubsection{Orthogonal polynomials}\label{sub:op}
For any $n$, define a sequence of polynomials
\[ P_l^n(x) = x^l +\text{lower order terms} \]
by the relations
\begin{equation}\begin{split}
&P_0^n(x) = 1~, \quad P_1^n(x) = x~, \quad P_2^n(x) = x^2 - n~, \\
&P_l^n(x) =  x P_{l-1}^n(x) - (n-1) P_{l-2}^n(x) \quad \text{for} \quad l \ge 3~.
\end{split}\end{equation}
Let  $ U_0, U_1, U_2, \cdots$  be the Chebyshev polynomials of the second kind defined via \begin{equation*}
	\begin{split}
	U_l(\cos\theta) = \frac{\sin((l+1)\theta)}{\sin\theta}~.
	\end{split}
	\end{equation*} 
Then $P_l^n(x)$ can be explicitly written as
	\begin{equation}\label{eq:viacheb}
	P_l^{n}(x) = (n-1)^{l/2} \times \left[  U_l\left( \frac{x}{2\sqrt{n-1}} \right) - \frac{1}{n-1} U_{l-2}\left( \frac{x}{2\sqrt{n-1}} \right)   \right] 
	\end{equation}
with the convention $U_{-2} \equiv U_{-1} \equiv 0$. Equivalently,
	\begin{equation}\label{eq:viachebinv}
	U_l\left( \frac{x}{2\sqrt{n-1}} \right) = \frac{1}{(n-1)^{l/2}}
	\sum_{k=0}^{\lfloor l/2\rfloor} P_{l-2k}^n(x)~.
	\end{equation}

We remark that $P_l^n$ are exactly the orthogonal polynomials with respect to the Kesten--McKay measure 
\[ \frac{n}{2\pi} \frac{\sqrt{4(n-1)-x^2}}{n^2 - x^2} dx~,\]
whereas $U_l(x/2)$ are orthogonal with respect to the semicircle measure.

\subsubsection{Enumeration of reduced paths}

The polynomials $P_l^n$ are convenient for enumerating reduced paths
on $n$-regular graphs. The following lemma may be seen as an instance
of this general principle; in our case, the graph is the Cayley graph of $S_n$
with respect to the generators $(k \, n)$, $1 \leq k \leq n-1$.

\begin{lemma}\label{l:cheb}
For any $n \geq 1$ and $l \geq 0$,
\[ P_l^{n-1}(X_n) = \sum_{\substack{1 \leq a_1, \cdots, a_l < n \\ a_j \neq a_{j+1}}} (a_1 n) (a_2 n) \cdots (a_l n)~.\]
\end{lemma}

\begin{proof} The statement is verified directly for $l=0,1,2$, and the case of
general $l$ follows by induction.
\end{proof}

We mention the analogous identity in the context of random matrices (cf.\ \cite{Sodin2007,Sodin2015} and references therein). Let $H^{(N)}=(H(i,j))_{i,j=1}^N$ be an Hermitian $N \times N$ matrix such that 
\[ H(i, i) = 0~, \quad |H(i, j)| = 1 \quad (i \neq j)~.\]
Then
\begin{equation}\label{eq:cheb_rm} P_l^{N-1}(H)(i, j) = \sum_{\substack{1 \leq a_0 = i, a_1, \cdots, a_{l-1},a_l=j \leq N \\ a_j \neq a_{j+1}, a_{j+2}}} H(a_0, a_1) H(a_1, a_2) \cdots H(a_{l-1}, a_l)~.\end{equation}

\subsubsection{Mixed moments}

Fix $n \geq 1$. Define the collection of $k$-tuples of lists \begin{equation*}
\begin{split}
\Sigma(m_1, t_1; \cdots; m_k, t_k) = \{ (\ell^1,\cdots,\ell^k) : \sigma(\ell^1) \sigma(\ell^2) \cdots \sigma(\ell^k) = 1\}
\end{split}
\end{equation*}
where the $p$-th list takes the form \begin{equation*}
\ell^p \in  \{ (a^p_1,\cdots,a^p_{m_p}): 1 \leq a^p_i \leq n - t_p - p  \}~,
\end{equation*}
and to each list $\ell^p$ we associate a permutation
\begin{equation*}
\sigma(\ell^p) = ( a^p_1 \,\, n - t_p - p + 1)\cdots (a^p_{m_p}\,\, n - t_p - p + 1)~.
\end{equation*}
Similarly, we consider a subcollection 
\[ \Sigma'(m_1,t_1;\cdots;m_k,t_k) \subset \Sigma(m_1, t_1; \cdots; m_k, t_k) \]
of lists that satisfy
\begin{equation*}
a^p_{i_p} \ne a^p_{i_p+1} \mbox{ for all } 1\le p \le k \mbox{ and } 1 \le i_p \le m_p-1~.
\end{equation*}

\begin{lemma}\label{lem:chebyshev_g}
	We have \begin{equation}
	\frac{1}{n!}\tr\left( \prod_{p=1}^k P^{n-t_p-p}_{m_p}(X_{n-t_p-p+1})  \right)   = |\Sigma'(m_1, t_1; \cdots; m_k, t_k)|.
	\end{equation}
\end{lemma}

\begin{proof} According to Lemma~\ref{l:cheb}
\[ P^{n-t_p-p}_{m_p}(X_{n-t_p-p+1})=
\sum_{\substack{1 \leq a_1, \cdots, a_{m_p} < n-t_p-p+1 \\ a_j \neq a_{j+1}}} (a_1 \,\,n-t_p-p+1) \cdots (a_{m_p}\,\,n-t_p-p+1)~; \]
multiply these identities, take the trace in the left regular representation, and
recall that for a permutation $\pi$
\[ \tr \pi = \begin{cases} n!~,&\pi = 1 \\
0~, &\pi \neq 1~.\end{cases}\]
\end{proof}

An analogue of Lemma~\ref{lem:chebyshev_g} in random matrix context
is as follows. Let $H$ be a random matrix such that
\begin{equation*}
H(i,j) \sim \begin{cases} \mathrm{unif}(S^1)~, & i \ne j~, \\ 
0~, & i = j~. \end{cases}
\end{equation*}
Then
\begin{equation}\label{eq:MMM}
\mathbb{E} \prod_{p=1}^k \tr P_{m_p}^{N-t_k-1} (H^{(N-t_k)}) = |\Sigma'_\text{RM}(m_1, t_1;
\cdots; m_k, t_k)|~,
\end{equation}
where $\Sigma'_\text{RM}(m_1, t_1; \cdots; m_k, t_k)$ is the collection of $k$-tuples of paths of the form 
\begin{equation*}
\gamma = \{ \gamma_1,\cdots,\gamma_k \}~, \quad \gamma_p = a^p_0 a^p_1 \cdots a^p_{m_p}~, \quad 1 \leq a^p_j \leq N - t_p
\end{equation*}
such that $a_{m_p}^p = a_0^p$, $a_j^p \neq a_{j+1}^p, a_{j+2}^p$ and 
for any $a \ne a'$, \begin{equation*}
	|\{ (p,j) : a^p_j = a, a^p_{j+1} = a' \}| = |\{ (p,j) : a^p_j = a', a^p_{j+1} = a \}|~.
	\end{equation*}

The elements of $\Sigma'_\text{RM}$  can be classified by the family of $k$-diagrams (see Definitions~\ref{def:diag} and \ref{def:kdiag}), which are graph-theoretic constructs retaining only the essential topological information of a given $k$-tuple of paths. In the next section, we prove a similar classification for elements of $\Sigma'$. 

\subsection{Some combinatorial constructions}

Our goal is to group the elements of $\Sigma'(m_1,t_1; \cdots; m_k, t_k)$ into equivalence
classes. The
idea is as follows. Assuming $k = 1$ and $t_1 = 0$, suppose that we are given a solution to the equation
\[ (j_1\, n) (j_2\,n) \cdots (j_k\, n) = 1~. \]
Consider the sequence of permutations
\[ \pi_0 = 1, \quad \pi_1 = (j_1 n)~, \quad \pi_2 = (j_1 n) (j_2 n)~, \quad \pi_k = (j_1 n) (j_2 n) \cdots (j_k n) = 1~.\]
We can decompose any permutation $\pi$ as a product of cycles
\[ \pi =c^* \prod_r c_{r}~, \]
where $c^*$ is the cycle containing $n$. When we pass from $i$ to $i+1$, one
of the following may happen. The first, topologically trivial, possibility is that the cycle structure does not change, and only $c^*$ increases or decreases by one. The second
possibility is that either $c^*$ splits into two cycles, or a cycle $c_r$ merges with $c^*$.
The equivalence classes which we define will keep track of the splitting and
merging of the cycles of $\pi_i$.

\subsubsection{Associating paths to lists of transpositions}

To define the equivalence classes, we first associate a path to any list of transpositions  $(a_1 n), (a_2 n), \cdots, (a_l n)$.

\begin{definition}[Associating a path]\label{def:assoc}
	For each $1 \le i \le n$, we prepare two vertices $i(w)$ and $i(b)$ (called `white' and `black' vertices of index $i$, respectively). Given a list $(a_1,\cdots,a_l)$, we construct a path of length $2l$ on this collection of $2n$ vertices, inductively on $l$. During the procedure, if we happen to add an arrow between two vertices where an arrow of the opposite direction had previously drawn, we consider them as being matched and invisible during the next steps of the procedure. 
	\begin{itemize}
		\item In the case $l = 1$, we draw an arrow from $n(w)$ to $a_1(b)$, and then from $a_1(b)$ to $a_1(w)$. This length 2 path is denoted by $P(a_1)$. 
		\item Given a path $P(a_1,\cdots,a_{l-1})$ associated with $(a_1,\cdots,a_{l-1})$, first draw an arrow from $a_{l-1}(w)$ (which is the last vertex of $P(a_1,\cdots,a_{l-1})$) to $a_l(b)$. It can be shown inductively that there is at most one unmatched incoming arrow to $a_{l}(b)$, not including the one which we just drew. If such an arrow does not exist, simply add an arrow to $a_l(w)$. Otherwise, add an arrow from $a_l(b)$ to match the unique incoming arrow to $a_l(b)$. This defines $P(a_1,\cdots,a_l)$. 
	\end{itemize}
\end{definition}

\begin{figure}
	\includegraphics[scale=1.5]{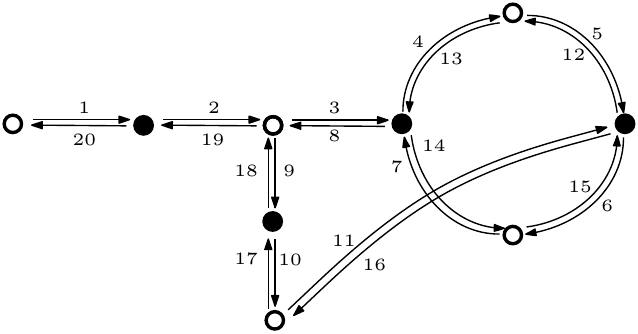}
	\centering
	\caption{The path corresponding to the list $(a,b,c,b,d,c,b,c,d,a)$. In the random matrix setting such a path (with coloring ignored) would correspond to $(\alpha, \beta, \gamma, \delta, \epsilon, \zeta, \eta, \delta, \gamma, \theta, \iota, \zeta, \epsilon, \delta, \eta, \zeta, \iota,\theta,\gamma,\beta,\alpha)$.}\label{fig:coxeter} 
\end{figure}

\begin{example}
Consider the list $(a,b,c,b,d,c,b,c,d,a)$, where $a,b,c,d$ are four distinct letters in $\{1,...,n-1\}$. We have the equality
	\begin{equation*}
	(a n)(b n)(c n)(b n)(d n) (c n)(b n)(c n)(d n)(a n)=1
	\end{equation*} in $S_n$. The associated path $P$ in $\Sigma'(20)$ is illustrated in Figure \ref{fig:coxeter}.
\end{example}

More generally, given $\overline{t} = (t_1,\cdots,t_k)$ and a $k$-tuple of lists $(\ell^1,\cdots,\ell^k)$, where $\ell^p$ is a list of transpositions $(a_j^p\,\, n-t_p - p +1)$ with $a_j^p < n-t_p - p +1$,  we associate a $k$-tuple of paths, as follows. 
First, we obtain the path $P_1$ corresponding to $\ell^1$ as in Definition~\ref{def:assoc}. 
On the $p$-th step, we associate a path $P_p$ to $\ell^p$ using the same construction,
except that this path starts on the white vertex $(n-t_p-p+1)(w)$, and
we allow an arrow from a path to be  matched with arrows from itself 
or any of the previous ones $P_1, \cdots, P_{p-1}$.

\begin{example}
	Let $k = 2$ and $\overline{t} = (0,0)$. Consider the relation
	\begin{equation*}
	(a \,n )(b\, n)(a\, n)(b \, n-1)(a \, n-1)(b \, n-1)=1
	\end{equation*}
	for some $a \ne b$,  $a,b\in \{1, \cdots, n-2\}$. The corresponding pair of paths is depicted in Figure \ref{fig:coxeter2}.
	\end{example}

\begin{figure}
	\includegraphics[scale=1.5]{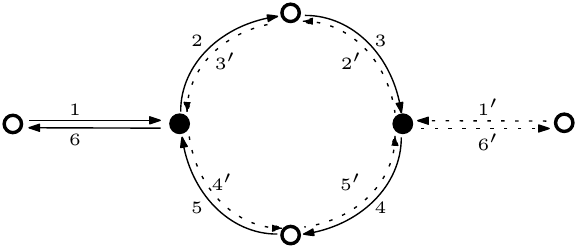}
	\centering
	\caption{Two paths corresponding to the pair of lists $(a,b,a),(b,a,b)$.  In the random matrix setting, this pair would correspond to $(\alpha,\beta,\gamma,\delta,\epsilon,\beta,\alpha), (\zeta,\delta,\gamma,\beta,\epsilon, \delta,\zeta)$.
}\label{fig:coxeter2}
\end{figure}

The main property of the construction is that the unmatched paths encode the
cycle structure of the product of transpositions. By construction, the unmatched 
part of the graph corresponding to a $k$-tuple of lists has vertices of degree $1$ 
and $2$. Therefore it is a union of cycles and intervals (`threads').  

\begin{lemma}\label{l:defworks} Let $(\ell^1, \cdots, \ell^k)$ be a $k$-tuple of lists, 
$\ell^p = (a_1^p, \cdots, a_{m_p}^{p})$,
and let $(P_1, \cdots, P_k)$ be associated $k$-tuple of paths. Let $P^*$
be the unmatched part of $(P_1, \cdots, P_k)$. Then
\begin{enumerate}
	\item the cycles of $P^*$ are in one-to-one correspondence with the cycles of
the product
\[\begin{split}
&(a_1^1 \,\, n-t_1) (a_2^1 \,\, n-t_1) \cdots(a_{m_1}^1 \,\, n-t_1) \cdots\\
&\quad (a_1^p \,\, n-t_p-p+1) (a_2^p \,\, n-t_p-p+1) \cdots(a_{m_p}^p \,\, n-t_p-p+1)\cdots\\
&\quad(a_1^k \,\, n-t_k-k+1) (a_2^k  \,\, n-t_k-k+1) \cdots(a_{m_k}^k \,\, n-t_k-k+1)
\end{split}\]
not containing the vertices $n-t_p -p+1$, and the intervals (`threads') are
in one to one correspondence with the cycles containing the special vertices.
	\item The collection $\Sigma(m_1, t_1; \cdots; m_k, t_k)$ consists precisely of lists for which $P^* = \varnothing$.
	
	\item The collection $\Sigma'(m_1, t_1; \cdots; m_k, t_k)$ consists precisely of lists 
for which $P^* = \varnothing$, and whose associated paths are non-backtracking; that is, two consecutive arrows are not matched with each other. 
\end{enumerate}
\end{lemma}

\begin{proof}[Proof of Lemma~\ref{l:defworks}]
The proof of the first item  proceeds by induction in the number of transpositions
in the lists. The path corresponding to one transposition $(j \, n)$ is $n(w) \to j(b) \to j(w)$. 
It is unmatched, and has one thread, which corresponds to the unique cycle of the transposition  $(j \, n)$. This provides the base of induction. For the induction step,
one verifies that a step of Definition~\ref{def:assoc} exactly corresponds to multiplication
of a permutation by a transposition from the right.

The second item follows from the first one: a transposition is the identity
if and only if it has no cycles.

Regarding the third item, one direction is clear: if there is a pair of adjacent letters that coincide, the associated path should have a backtracking arrow. The other implication can be easily proved by an induction on the length.
\end{proof}

\subsubsection{Diagrams}

We introduce the notion of a diagram. We first consider the case $k = 1$ and $t_1 = 0$.  We copy the following definition from \cite[Definition II.1.3]{FS}. 

\begin{definition}[Diagram]\label{def:diag}
	A diagram consists of a graph $G = (V,E)$ together with a circuit $p=v_0 v_1\cdots v_r v_0$ on $G$, satisfying the following conditions. \begin{itemize}
		\item The circuit $p$ is non-backtracking (i.e. $v_j$ is not the reverse of $v_{j+1}$ for $0 \le j \le r$ with $v_{r+1} = v_0$).
		\item For each edge $(v,v') \in E$, \begin{equation*}
		|\{ j : v_j = v, v_{j+1} = v' \}| = | \{ j: v_j = v', v_{j+1} = v \}| = 1~.
		\end{equation*}
		\item The degree of each $v_j$ is 3, except for that of the initial vertex $v_0$, which is 1.
	\end{itemize}
	If there is a function $\overline{l}:E \rightarrow \{ 0,1,2,\cdots \}$ which assigns a length to each edge of $G$, we call $G$ a metric diagram.
\end{definition}

In the definition below, we illustrate the steps to `contract' a given path in $\Sigma'(2m,0)$ to a metric diagram. 

\begin{figure}
	\includegraphics[scale=1.5]{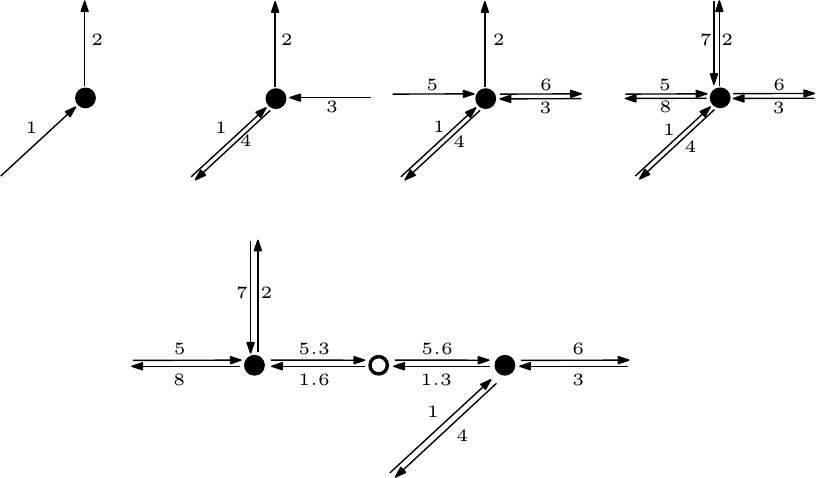}
	\centering
	\caption{Appearance of a high-degree vertex (above), reducing the degree to 3 (below).}\label{fig:resolve} 
\end{figure}

\begin{definition}[The contraction map]\label{def:contraction}
	Take a path $P(a_1,\cdots,a_{2m})  \in \Sigma'(2m,0)$. At each stage, there is a natural way to modify the path, which becomes the circuit in the resulting diagram.
	\begin{enumerate}
		\item Collapse each pair of matched arrows into a single undirected edge. 
		\item(Regularizing the initial vertex) If the vertex  $n(w)$ has degree greater than 1, pick a letter $1 \le a < n$ which does not appear in $\{a_1,\cdots,a_{2m}\}$ and draw an edge between $n(w)$ and $a(b)$, and also between $a(b)$ and $a(w)$. 
		\item(Removing multiplicity of vertices) Take a black vertex $a(b)$ involved in the path, and take the smallest index $1\le j$ for which all arrows (nonempty set) connected to $a(b)$ are matched, in the partial path $P(a_1,\cdots,a_j)$. If there are additional occurrences of the letter $a$ after $i > j$, prepare a new black vertex $a'(b)$ and replace each occurrences of $a(b)$ by $a'(b)$ after $i > j$. Here $1 \leq a' < n-t_1$ is some letter  which does not appear in $\{a_1,\cdots,a_{2m}\}$. We repeat this procedure until we do not need to prepare any new black vertices. 
		
		Then, we apply the above procedure for all white vertices appearing in the path.
		\item(Reducing the degree) Assume that a black vertex has degree exceeding 3. It is not hard to see that such a high-degree vertex should be obtained by a repetition of the process depicted in Figure \ref{fig:resolve} (above). In view of this, we only consider the case when degree is 4. For this, we prepare a pair of new black and white vertices and modify the path as shown in Figure \ref{fig:resolve} (below). 
		
		As before, we repeat the procedure for all white vertices of degree exceeding 3. 
		\item Collapse all vertices of degree 2, and define $l(e) = 0$ if $e$ is an edge created in the above procedure. Otherwise, we set $l(e)$ as the number of collapsed vertices on that edge plus one. 
	\end{enumerate}
\end{definition}

We note that while the diagrams in \cite{FS} did not have colored vertices, the coloring can be recovered from the circuit in a unique way. 

\begin{example}
	The path $P(a,b,c,b,d,c,b,c,d,a)$ depicted in Figure \ref{fig:coxeter} is associated with the metric diagram in Figure \ref{fig:diagram}.
\end{example}

\begin{figure}
	\includegraphics[scale=1.5]{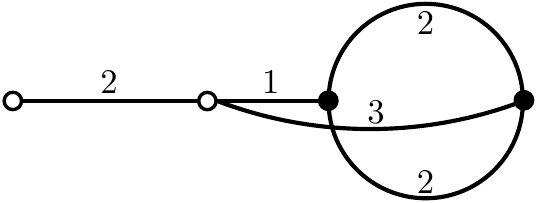}
	\centering
	\caption{The metric diagram associated to the path $P(a,b,c,b,d,c,b,c,d,a)$, without the circuit}\label{fig:diagram} 
\end{figure}

In \cite[Section II.2]{FS}, an automaton which generates all possible diagrams is presented. To explain the ideas, let a particle travel through the circuit and consider its trail, erasing the parts which have been passed twice (in opposite directions). Then, at each moment of time, the trail consists of a thread (starting at the initial vertex) and a number of loops. Therefore, there are two types of transitions that the trail goes through, `creation' of a new loop and `annihilation' of an existing loop. We let $s>0$ be the number of transitions associated with a diagram, which must be an even integer.

\begin{lemma}[see Claim II.2.1 of \cite{FS}]\label{lem:counting_diagrams}
	If a diagram is generated by $s$ transitions, it has $3s-1$ edges and $2s$ vertices. Denoting $D(s)$ as the number of such diagrams, we have estimates \begin{equation}
	(s/C)^s \le D(s) \le C^{s-1} s^s~,
	\end{equation}
	for some absolute constant $C > 0$.
\end{lemma}

At this point, we make the following observation (cf.\ Claim II.1.4 of \cite{FS}).

\begin{lemma}\label{lem:degree_of_freedom} Given a metric diagram, there are at most \begin{equation*}
	(n-1)^{m-s/2} = (n-1)^{\sum_{e\in E} l(e) /2 - s/2}
	\end{equation*}
	elements of $\Sigma'(2m,0)$ corresponding to it. If $l(e) \ge 1$ for every $e \in E$ then there are exactly \begin{equation*}
	(n-1)(n-2)\cdots (n-(m-s/2))
	\end{equation*}
	such elements. In particular, if in addition we have $m-s/2 = o(n^{1/2})$ then the number is \begin{equation*}
	n^{m-s/2}(1+o(1)).
	\end{equation*}
\end{lemma}

Before the proof, we quote the following known fact (which can be also proved
by the graphical methods developed here):
\begin{lemma}\label{l:cox}
The symmetric group $S_n$ is generated by the transpositions 
$(1\, n)$, $(2\, n)$, \dots, $(n-1\, n)$,
with the trivial relations $(c\,n) (c\,n) = 1$,  the commutation relations 
	\begin{equation*}
	\begin{split}
	(c\,n) (b\,n) (a_1 \,n)\cdots (a_m \,n) (b\,n) = (b\,n) (a_1 \,n)\cdots (a_m \,n) (b\,n) (c\,n)~, \quad c \notin \{ b,a_1,\cdots,a_m \}~,
	\end{split}
	\end{equation*} 
and the Coxeter relations 
$$(a\,n)(b\,n)(a\,n)(b\,n)(a\,n)(b\,n) =1~.$$ 
\end{lemma}

\begin{proof}[Proof of Lemma~\ref{lem:degree_of_freedom}]
Consider a group equation of the form \begin{equation*}
	(a_1 \, n)(a_2 \, n)\cdots(a_{2m} \, n) = 1~.
	\end{equation*}
Any solution can be brought to the form $1 = 1$ using the relations of
Lemma~\ref{l:cox}. Using a Coxeter relation reduces the degree of freedom to choose letters from $\{1,\cdots,n-1\}$ by 1. On the other hand, the number of Coxeter relations used can be seen from the associated diagram as the number $s/2$ of pairs of creation and annihilation of loops. Indeed, each creation step corresponds to opening up the left half of a Coxeter relation, while there should be an annihilation step corresponds to closing it. The order in which the relations are used is encoded by the metric diagram.
This proves the upper bound.

Regarding the second item, the number $(n-1)(n-2)\cdots (n-(m-s/2))$ gives a lower bound, as we can pick $m-s/2$ distinct letters from $\{1,\cdots,n-1\}$ and assemble a path on top of the given diagram. Assume the number of distinct letters used in a path is strictly less than $m-s/2$. Then there should exist a letter, say $c$, appearing at least 4 times. Then the vertex $c(b)$ has degree at least 4, and therefore the associated diagram cannot have strictly positive metric function. 
\end{proof}

\begin{remark}
	
	The associated diagram visualizes the way that the Coxeter relations are `embedded' between the trivial relations in a solution of $(a_1 \, n)(a_2 \, n)\cdots(a_{2m} \, n) = 1$. The metric equals the number of the above commutation relations that needs to be used on each edge of the diagram, minus one.  
\end{remark}

\begin{remark}
	Yet another interpretation can be given for the number $s$ of a diagram: $s = 2g$ where $g$ is the genus of the surface obtained by gluing the boundary of a disc according to the circuit of the diagram. For the case of the diagram described in Figure~\ref{fig:diagram}, we obtain a torus; see Figure~\ref{fig:torus}. Different diagrams 
	correspond to homotopically distinct ways to obtain a compact, orientable surface of genus $g$ from gluing the boundary of a disc with a marked point. 
	Similar remarks hold for $k$-diagrams to be defined below (see \cite[Section 3.2.1]{Sodin2014}).
\end{remark}

\begin{figure}
\begin{center}
\includegraphics{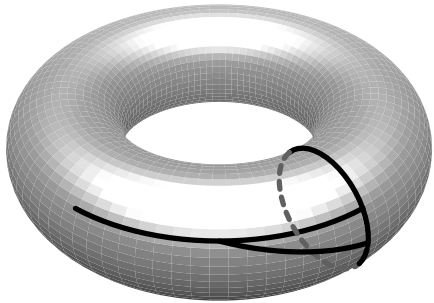}
\end{center}
\caption{The surface obtained from the diagram in Figure \ref{fig:diagram}, cf.\  \cite[Figures 6--7]{Ok}.}\label{fig:torus}
\end{figure}

\medskip

The notion of a diagram, the contraction map, and the properties described above carry over to the case $ k \ge 1$ and any $\overline{t} = (t_1,\cdots,t_k)$. This extension is covered in \cite[Section II.3]{FS}. Let us simply copy the definition, which is \cite[Definition II.3.1]{FS}.

\begin{definition}[$k$-Diagram]\label{def:kdiag}
	A $k$-diagram consists of a graph $G = (V,E)$ together with a $k$-tuple of circuits \begin{equation*}
	\overline{p} = v^1_0 v^1_1\cdots v^1_{r_1} v^1_0~, \cdots,   v^k_0 v^k_1 \cdots v^k_{r_k} v^k_0 
	\end{equation*} on $G$, satisfying \begin{itemize}
		\item each of the $k$ circuits is non-backtracking,
		\item for each edge $(v,v') \in E$, \begin{equation*}
		|\{ (p,j) : v^p_j = v, v^p_{j+1} = v' \}| = | \{ (p,j): v^p_j = v', v^p_{j+1} = v \}| = 1~, 
		\end{equation*}
		\item the degree of each $v^p_j$ is 3, except for those of the initial vertices $v_0^p$, which is 1.
	\end{itemize}
	If there is a function $\overline{l}:E \rightarrow \{ 0,1,2,\cdots \}$ which assigns a length to each edge of $G$, we call $G$ a metric diagram.
\end{definition}

We omit the definition of the contraction map (from $\Sigma'(m_1,t_1;\cdots;m_k,t_k)$ to the set of $k$-diagrams), as it can be constructed exactly as in the case $k = 1$. 

Any $k$-diagram can be generated by a sequence of creation and annihilation steps, as in the $k = 1$ case. The only difference is that the particle returns to its original position exactly $k$ times during the entire procedure. Therefore, a $k$-diagram is associated with an even integer $s > 0$, and we have

\begin{lemma}[see Claim II.3.2 of \cite{FS}]\label{lem:counting_diagrams_k}
	If a $k$-diagram is generated by $s$ transitions, it has $3s-k$ edges and $2s$ vertices. Denoting $D_k(s)$ as the number of such diagrams, we have estimates \begin{equation}
	(s/C)^s/(k-1)! \le D_k(s) \le C^{s-1} s^s/(k-1)!~,
	\end{equation}
	for some absolute constant $C > 0$.
\end{lemma}

Let us state the extension of Lemma \ref{lem:degree_of_freedom} to this case. For each edge $e \in E$, we set $1 \le p_{-}(e) \le p_{+}(e) \le k$ as the indices of the circuits traversing $e$. We also let $ n_0 = \min_{1 \le p \le k} n_p$. 

\begin{lemma}\label{lem:degree_of_freedom2}
	Given a metric $k$-diagram, there are at most \begin{equation*}
	\begin{split}
  (n_0-1) ^{-s/2}(n - 1)^{\kappa} \prod_{e \in E} \left( (n_{p_+(e)}-1) \wedge (n_{p_-(e)}-1) \right)^{\lceil l(e)/2 \rceil -1}
	\end{split}
	\end{equation*} elements of $\Sigma'(m_1,t_1;\cdots;m_k,t_k)$ corresponding to it, where $$0 \le \kappa := \sum_{e \in E} (l(e)/2 - \lceil l(e)/2 \rceil +1)\le 3s~.$$ If $l(e) \ge 1$ for every $e \in E$, then there are at least \begin{equation*}
	\begin{split}
	(n-1) ^{-s/2}(n_0 - m+s/2)^{\kappa} \prod_{e \in E} \left( (n_{p_+(e)}-m+s/2) \wedge (n_{p_-(e)}-m+s/2) \right)^{\lceil l(e)/2 \rceil -1}
	\end{split}
	\end{equation*} such paths, where we set $m = \sum_p m_p/2 = \sum_e l(e)/2$. In particular, if we have in addition $m, s = o(n^{1/2})$ and $n_0 = n(1+o(1))$ then the number is \begin{equation*}
	\begin{split}
	n^{-s/2} \prod_{e \in E} \left(  n_{p_+(e)} \wedge n_{p_-(e)}  \right)^{l(e)/2} (1+o(1))~. 
	\end{split}
	\end{equation*}
\end{lemma}

\begin{proof}
	The number of black vertices in a path of $\Sigma'(m_1,t_1;\cdots;m_k,t_k)$ contracting to a given metric $k$-diagram is $m-s/2$, as in Lemma \ref{lem:degree_of_freedom}. 
	For each edge $e$ of the $k$-diagram, the corresponding segment of the path has at least $\lceil l(e)/2 \rceil -1 $ black vertices, on which we can put letters from the set $\{ 1,\cdots, (n_{p_+(e)}-1)\wedge (n_{p_-(e)}-1) \}$. The upper and lower bounds follow. 
\end{proof}

\subsubsection{A regular subclass}

Now that we have established the passage from the collection of lists $\Sigma'(2m,0)$ to the set of diagrams, let us define a `regular' subcollection $\Sigma^\star(2m,0) \subset \Sigma'(2m,0)$ which behaves in a particularly nice way in this procedure. For paths in this subcollection, the steps 2--4 in Definition \ref{def:contraction} are unnecessary (which means that the associated metric is strictly positive), and a simple graph-theoretic characterization is available. 

Forgetting about the vertex labels, $\Sigma^\star(2m,0)$ consists of lists whose associated paths of length $2m$ satisfy the following conditions: \begin{itemize}
	\item start from a white vertex of degree 1,
	\item alternate between white and black vertices,
	\item are non-backtracking,
	\item close up without any unmatched arrows,
	\item no vertex has degree exceeding 3, and
	\item if we order all the vertices of degree 3 according to the first time the path visits them, the vertices alternate in color, starting with a black vertex.
\end{itemize}

A path $P$ satisfying the above conditions has $m-s/2$ black vertices for some even $s > 0$. If we pick $m-s/2$ distinct letters $c_1,\cdots,c_{m-s/2}$ from the set $\{ 2,\cdots,n \}$, then by appropriately labeling vertices of $P$ with $c_i(b)$ and $c_i(w)$, together with $1(w)$, we obtain \begin{equation*}
P = P(a_1,\cdots,a_{2m})~,\quad (a_1,\cdots,a_{2m}) \in \Sigma^\star(2m,0)~, \mbox{ where}\quad a_j \in \{ c_1,\cdots,c_{m-s/2} \}~.
\end{equation*}
We use this observation in the proof of Lemma \ref{lem:MMM}. Indeed, we will show with Lemma \ref{lem:degree_of_freedom} that in the asymptotic regime that we consider, we have \begin{equation*}
\Sigma^\star(2m,0) = \Sigma'(2m,0) (1 + o(1))~.
\end{equation*}

Similarly as before, for each $k > 1$, we define a subclass \begin{equation*}
\Sigma^\star(m_1,t_1;\cdots;m_k,t_k) \subset \Sigma'(m_1,t_1;\cdots;m_k,t_k)
\end{equation*}
as the collection of lists whose associated metric $k$-diagrams have strictly positive metric. A totally analogous graph-theoretic characterization for $\Sigma^\star(m_1,t_1;\cdots;m_k,t_k)$ can be given, and again we will need the fact that \begin{equation*}
\Sigma^\star(m_1,t_1;\cdots;m_k,t_k) = \Sigma'(m_1,t_1;\cdots;m_k,t_k) ( 1 + o(1))
\end{equation*}
in the asymptotic regime considered in Lemma \ref{lem:MMM}.

\subsubsection{Limiting continuous functions}\label{s:cont}

Now we describe the functions $\phi$ which serve as the Laplace transform
of the correlation functions of the Airy line ensemble, (\ref{eq:phi1}). We start
from the description of an auxiliary set of functions $\psi$.

Fix a $k$-diagram $\mathcal{D}$ with $|E|= 3s-k$ and $|V| = 2s$. For $1 \le p \le k$ and $e \in E$, let $c_p(e) \in \{ 0,1,2 \}$ be the number of times $e$ is traversed by the $p$-th circuit.  Given a vector $\overline{\alpha} \in \mathbb{R}^k_+$, consider the following system of equations $S_\mathcal{D}(\overline{\alpha})$ with variables $\xi(e),e\in E$: \begin{equation*}
\sum_{e \in E} c_p(e) \xi(e) = \alpha_p, \quad 1 \le p \le k.
\end{equation*}
Then,  we denote $\Delta_\mathcal{D}(\overline{\alpha})$ by the convex polytope of positive real solutions to $S_\mathcal{D}(\overline{\alpha})$. This is a $(3s-2k)$-dimensional polytope in $\mathbb{R}^{3s-k}$. We normalize the Lebesgue measure on $\Delta_\mathcal{D}(\overline{\alpha})$ by the factor 
\begin{equation}\label{eq:norm}
\begin{split}
\lim_{ \overline{m}\rightarrow \infty} \frac{\left|  \Delta_\mathcal{D}(\overline{m}) \cap \mathbb{Z}^{3s-k}  \right|}{\mathrm{Vol}( \Delta_\mathcal{D}(\overline{m}))}
\end{split}
\end{equation}
 where the limit is taken along integer vectors with even components $\overline{m}$ asymptotically parallel to $\overline{\alpha}$, and $\mathrm{Vol}(\cdot)$ is taken with respect to the standard $(3s-2k)$-dimensional Lebesgue measure. We then define the integral \begin{equation}
I^{\mathcal{D}}(\overline{\alpha},\overline{\tau})  = \int_{\Delta_{\mathcal{D}}(\overline{\alpha})} \exp\left( -\sum_{e \in E}  |\tau_{p_+(e)} - \tau_{p_{-}(e)}|  \omega(e)  \right) d\overline{\omega}
\end{equation}
 where $d\overline{\omega}$ denotes the integration with respect to the above normalized measure. These integrals allow us to define a function \begin{equation*}
 \begin{split}
 \psi: \uplus_{k\ge 0} \left( \mathbb{R}^k_+ \times \mathbb{R}^k_+ \right) \rightarrow \mathbb{R}_+
 \end{split}
 \end{equation*} inductively on $k \ge 0$ by \begin{equation*}
 \begin{split}
 \sum_{\mathcal{D}} I^\mathcal{D}(\overline{\alpha},\overline{\tau}) = \sum_{I \subset \{1,\cdots,k\}} \psi(\overline{\alpha}|_I,\overline{\tau}|_I) \psi(\overline{\alpha}|_{I^c},\overline{\tau}|_{I^c})
 \end{split}
 \end{equation*} (the sum on the left hand side ranges over all possible $k$-diagrams $\mathcal{D}$), with the convention $\psi(\varnothing,\varnothing)\equiv 1$. It will follow from the proof of Lemma \ref{lem:MMM} that $\psi(\overline{\alpha},\overline{\tau})$ is a continuous function $\mathbb{R}^k_+ \times \mathbb{R}^k_+  \rightarrow \mathbb{R}_+$, for each $k \ge 1$. 

\begin{example}\label{ex:k=1}
	In the simplest case when $k = 1$ and $\tau = 0$, we may compute directly that \begin{equation*}
	\begin{split}
	I^{\mathcal{D}}(\alpha,0) = \frac{1}{(3s-2)!} \left( \frac{\alpha}{2} \right)^{3s-2}
	\end{split}
	\end{equation*} for any 1-diagram $\mathcal{D}$ with $3s-1$ edges. Hence we obtain \begin{equation*}
	\begin{split}
	\psi(\alpha,0) = \frac{1}{2}\sum_{s} \frac{D(s)}{(3s-2)!} \left( \frac{\alpha}{2} \right)^{3s-2}
	\end{split}
	\end{equation*} with $D(s)$ being the number of 1-diagrams with $3s-1$ edges. 
\end{example}

Finally, define a function \begin{equation*}
\begin{split}
\phi: \uplus_{k\ge 0} \left( \mathbb{R}^k_+ \times \mathbb{R}^k_+ \right) \rightarrow \mathbb{R}_+
\end{split}
\end{equation*} via the formula 
\begin{equation}\label{eq:phidef}
\phi(\overline{\alpha},\overline{\tau}) = \sum_{I \subset\{1,\cdots,k\}}  \prod_{p\notin I} \left(\frac{1}{2\sqrt{\pi}} \left( \frac{\alpha_p}{2} \right)^{-3/2}  \right)  \prod_{p \in I} \left( \int_{ \mathbb{R}_+ }  \frac{2\xi_p d\xi_p}{\sqrt{\pi \alpha_p}}   \psi( 2 \sqrt{\overline{\alpha}}|_I \cdot \overline{\xi}|_I,\overline{\tau} ) \right),
\end{equation} where $\sqrt{\overline{\alpha}}|_I \cdot \overline{\xi}|_I$ denotes the coordinate-wise product; $\phi(\varnothing,\varnothing) \equiv 1$.

It is shown in \cite{Sodin2015}  that for $N_p = N - t_p'$, $t_p'/N^{2/3} \to \tau_p$, $r_p' / N^{2/3} \to \alpha_p$
\begin{equation}\label{eq:fix}\mathbb{E} \prod_{p=1}^k \tr \left( \frac{H^{(N_p)}}{2\sqrt{N_p}}\right)^{r_p'} - 
\sum_{I \subset \{1,\cdots,k\}} (-1)^{\sum_{i\in I} r_i} \phi(\overline{\alpha}|_I,\overline{\tau}|_I) \phi(\overline{\alpha}|_{I^c},\overline{\tau}|_{I^c}) \to 0 \end{equation}
and that  the functions $\phi$ are the Laplace transforms of the Airy$_2$ line ensemble; 
\begin{equation}\label{eq:phi=lapl}
\begin{split}
\phi(\overline{\alpha},\overline{\tau}) = \mathbb{E} \prod_{p=1}^k \sum_j \exp(\alpha_p \lambda_j(\tau_p)).
\end{split}
\end{equation} 

\begin{remark}\label{rem:fix} We use the opportunity to correct several mistakes from \cite{Sodin2015}.
First, eq.\ (3.8) thereof does not take proper account of the first term of (\ref{eq:phidef}).
The correct definition of $\phi^\#$ should be
\begin{equation}\phi^\#(\bar\alpha,\bar{s}) = \sum_{I \subset \{1,\cdots,k\}} \phi(\bar\alpha|_I, \bar{s}|_I) \phi(\bar\alpha|_{I^c}, \bar{s}|_{I^c})~, \end{equation}
where $\phi$ is defined as in (\ref{eq:phidef}). Second, parity is not properly taken into account in \cite[Lemma~3.1 and Lemma~3.2]{Sodin2015}, which hold as stated only if all $m_p$ and $n_p$ are even. The corrected formulation is given in (\ref{eq:fix}) 
and (\ref{eq:fix2}) below. Third, the normalization
factor (\ref{eq:norm}) is not explicitly stated in \cite{Sodin2015}, and neither it is in \cite{Sodin2014}.
\end{remark}

\section{The main technical statement}\label{S:main}

We are ready to state and prove the main technical lemma. To begin with, we prove the
partition analogue of (\ref{eq:fix}) for the (usual) moments 
\begin{equation}\label{eq:JMtraces'}
\mathcal{M}(\overline{r},\overline{t}) = \frac{1}{n! \prod_{p=1}^k n_p^{-1/2}} \tr \prod_{p=1}^k \left( \frac{X_{n_p}}{2n_p^{1/2}}  \right)^{r_p}~,\quad n_p = n-t_p-p+1~.
\end{equation}

The following lemma will imply Proposition~\ref{prop'}, see \ref{s:prop'}. Note the
similarity with the random matrix analogue (\ref{eq:fix}).

\begin{lemma}\label{lem:MM}
	We have a bound 
	\begin{equation}\label{eq:MM_bound}
	\begin{split}
	\mathcal{M}(\overline{r},\overline{t})  \le \prod_{p=1}^k \frac{Cn^{1/2}}{r_p^{3/2}} \exp(C_k r_p^3/n)
	\end{split}
	\end{equation}
	as well as, in the asymptotic regime \begin{equation}\label{eq:JM_asymp'}
	r_p \sim  2\alpha_p n^{1/3}~,\quad t_p = 2\tau_p n^{5/6}
	\end{equation} we have
	\begin{equation}\label{eq:MM_asym}
	\begin{split}
	\mathcal{M}(\overline{r},\overline{t}) = \sum_{I \subset \{1,\cdots,k\}} (-1)^{\sum_{i\in I} r_i} \phi(\overline{\alpha}|_I,\overline{\tau}|_I) \phi(\overline{\alpha}|_{I^c},\overline{\tau}|_{I^c})  + o(1)~.
	\end{split}
	\end{equation}
\end{lemma}

\medskip
We prove this lemma by first showing a corresponding statement for the modified moments  involving Chebyshev polynomials and then `integrate' to recover the usual moments with the help of formulas \eqref{eq:Snyder}. Therefore, we define \begin{equation}\label{eq:MMM_perm}
\widetilde{\mathcal{M}}(\overline{m},\overline{t}) = \frac{1}{n!\prod_{p=1}^k n_p^{m_p/2-1/2} } \tr \left( \prod_{p=1}^k P_{m_p}^{n-t_p-p}(X_{n_p}) \right)~,\quad n_p = n-t_p-p+1~.
\end{equation}

\begin{lemma}\label{lem:MMM}
	We have an upper bound \begin{equation}\label{eq:upper_bound}
	\widetilde{\mathcal{M}}(\overline{m},\overline{t}) \le (Cm)^k \exp(C_k m^{3/2}n^{-1/4})~.
	\end{equation}
	
	Moreover, in the asymptotic regime \begin{equation}\label{eq:asym_perm}
	m_p \sim \alpha_p n^{1/6}~,\quad t_p = 2\tau_p n^{5/6}~,\quad n_p \sim n(1 - 2\tau_p n^{-1/6})~,
	\end{equation}
	we have 
	\begin{equation}\label{eq:asym_expression_perm}
	\frac{1}{n^{k/6}} \widetilde{\mathcal{M}}(\overline{m},\overline{t}) = \sum_{I \subset \{1,\cdots,k\}} (-1)^{\sum_{i\in I} m_i} \psi(\overline{\alpha}|_I,\overline{\tau}|_I) \psi(\overline{\alpha}|_{I^c},\overline{\tau}|_{I^c})  + o(1)~.
	\end{equation}
\end{lemma}

\noindent The counterpart of (\ref{eq:asym_expression_perm}) for random matrices is
\begin{equation}\label{eq:fix2} 
\mathbb{E} \prod_{p=1}^k \tr P_{m_p'}^{N_p-1}(H^{(N_p)})
= \sum_{I \subset \{1,\cdots,k\}} (-1)^{\sum_{i\in I} m_i} \psi(\overline{\alpha}|_I,\overline{\tau}|_I) \psi(\overline{\alpha}|_{I^c},\overline{\tau}|_{I^c})  + o(1)
\end{equation}
where $N_p = N - t_p'$, $t_p'/N^{2/3} \to 2\tau_p$, $m_p' / N^{1/3} \to \alpha_p$.

\subsection{Proof of Lemma \ref{lem:MMM}}
	
	As a warm-up, we establish the lemma in the simplest case, when $k = 1$ and $t = 0$. This proof is parallel to that given in \cite[Section I.5]{FS} for the random matrix case.
	
\subsubsection{The case of $k = 1$ and $t = 0$} We set $m_1 = 2m$ so that $2m/n^{1/6} \rightarrow \alpha$ in the limit $n\rightarrow\infty$. We need to establish \begin{equation*}
	\frac{1}{n! n^{m-1/2}} \tr P_{2m}^{n-1} (X_n) = n^{1/6} (2\psi(\frac{2m}{n^{1/6}},0) + o(1) )~,
	\end{equation*}
	together with an upper bound \begin{equation*}
	\frac{1}{n! n^{m-1/2}} \tr P_{2m}^{n-1} (X_n) \le Cm \exp(Cm^{3/2} n^{-1/4})~.
	\end{equation*}
	
	\medskip
	
	In view of the identity $\tr P_{2m}^{n-1}(X_n) = n! | \Sigma'(2m,0)|$, we estimate the contribution of each 1-diagram $\mathcal{D}$ to the set $\Sigma'(2m,0)$. Fix a diagram $\mathcal{D}$ with $2s$ vertices and $3s-1$ edges, and consider the subset of elements of $\Sigma'(2m,0)$ corresponding to $\mathcal{D}$. Forgetting about the vertex labels for a moment, denoting $q$ to be the number of edges of $\mathcal{D}$ which should carry an odd length (recall that there is a parity restriction for each edge due to vertex coloring), the number of ways to place the lengths on $\mathcal{D}$ does not exceed \begin{equation*}
	{{(2m+q)/2 + (3s-1) - 1}\choose{(3s-1)-1}} \le \frac{(m+6s-3)^{3s-2}}{(3s-2)!}
	\end{equation*}
	and then (Lemma~\ref{lem:degree_of_freedom}) a selection of $m-s/2$ distinct letters from the set $\{1,\dots,n-1\}$ specifies an element of $\Sigma'(2m,0)$. Hence we have (with Lemma~\ref{lem:counting_diagrams}) an upper bound \begin{equation*}
	\begin{split}
	|\Sigma'(2m,0)| &\le \sum_{2 \le s \le 2m} D(s) (n-1)^{m-s/2}\frac{(m+6s-3)^{3s-2}}{(3s-2)!}\\
	&\le m n^m \sum_{2 \le s } C^{s-1} s^s n^{-s/2} \frac{m^{3s-2}}{(3s-2)!}\\
	&\le m n^{m-1/2} \exp(Cm^{3/2}n^{-1/4})~.
	\end{split}
	\end{equation*}
	
	To obtain the asymptotic expression, we take $s_0$ such that \begin{equation*}
	1 \wedge m^{3/2}/n^{1/4} \ll s_0 \ll m^{1/2}
	\end{equation*}
	(where we write $A \ll B$ if $A/B \to 0$), and with both sides of the estimate from Lemma~\ref{lem:counting_diagrams}, one can see that the  contribution of $s_0 < s \le 2m$ to the sum \begin{equation*}
	\sum_{2 \le s \le 2m} D(s) (n-1)^{m-s/2} \frac{(m+6s-3)^{3s-2}}{(3s-2)!}
	\end{equation*}
	is negligible. Therefore, \begin{equation*}
	\begin{split}
	|\Sigma'(2m,0)| &\le \sum_{2 \le s \le 2m} D(s) (n-1)^{m-s/2} \frac{(m+6s-3)^{3s-2}}{(3s-2)!} \\
	&\le \sum_{2 \le s \le s_0} D(s) (n-1)^{m-s/2} \frac{(m+6s-3)^{3s-2}}{(3s-2)!}  (1 + o(1)) \\
	&\le \sum_{2 \le s \le s_0} D(s) (n-1)^{m-s/2} \frac{m^{3s-2}}{(3s-2)!}  (1 + o(1)) \\
	&\le \sum_{2 \le s} D(s) (n-1)^{m-s/2} \frac{m^{3s-2}}{(3s-2)!}  (1 + o(1))~.
	\end{split}
	\end{equation*}
	
	On the other hand, we note that for each diagram, there are \begin{equation*}
	{{(2m+q)/2 - 1}\choose{(3s-1)-1}}
	\end{equation*}
	ways to place the lengths with correct parity so that each edge has a strictly positive length. For $s \le s_0 \ll m^{1/2}$, then 
	\begin{equation*}
	{{(2m+q)/2 - 1}\choose{3s-2}} = \frac{m^{3s-2}}{(3s-2)!} (1+o(1))
	\end{equation*}
	and, by Lemma~\ref{lem:degree_of_freedom}, there are $(n-1)^{m-s/2} (1+o(1))$ ways to choose the vertices. Any path coming from this procedure is associated with an element of $\Sigma^\star(2m,0)$. Hence \begin{equation*}
	\begin{split}
	|\Sigma'(2m,0)| \ge |\Sigma^\star(2m,0)| &\ge \sum_{2 \le s \le s_0} D(s)(n-1)^{m-s/2} \frac{m^{3s-2}}{(3s-2)!} (1+o(1)) \\
	&\ge \sum_{2 \le s} D(s)n^{m-s/2} \frac{m^{3s-2}}{(3s-2)!} (1+o(1))
	\end{split}
	\end{equation*}
	which establishes the desired statement (recall Example~\ref{ex:k=1}).

\subsubsection{The general case}

Let us focus on the asymptotics \eqref{eq:asym_expression_perm};
the upper bound \eqref{eq:upper_bound} follows along similar lines. 
As in the case $k=1$, we start with the inequality
\[\frac{\Sigma'(m_1,t_1;\cdots;m_k,t_k)}{\prod_{p=1}^k n_p^{m_p/2-1/2}}  = \widetilde{\mathcal{M}}(\overline{m},\overline{t}) 
\geq \frac{\Sigma^\star(m_1,t_1;\cdots;m_k,t_k)}{\prod_{p=1}^k n_p^{m_p/2-1/2}}  \]
and prove that the right-hand side of \eqref{eq:asym_expression_perm}
is a lower bound for $\Sigma^*$ and an upper bound for $\Sigma'$. The two
estimates are proved similarly to one another, therefore we focus on the
second one. Let us estimate the contribution of each diagram $\mathcal{D}$.
	
	To begin with, we may rewrite the prefactor as \begin{equation*}
	\frac{1}{\prod_p n_p^{-1/2}} \prod_{e \in E} \left( n_{p_+(e)}^{-l(e)/4} \cdot n_{p_{-}(e)}^{-l(e)/4} \right)
	\end{equation*}
	and Lemma \ref{lem:degree_of_freedom2} gives that the combinatorial factor coming from the choice of letters equals \begin{equation*}
	(1 + o(1)) n^{-s/2} \prod_{e \in E}  \left( n_{p_+(e)} \wedge n_{p_{-}(e)}  \right)^{l(e)/2}~.
	\end{equation*}
	
	The metric $\overline{l}$ should satisfy the system of equations \begin{equation*}
	\sum_{e \in E} c_p(e) l(e) = 2m_p~, \quad 1 \le p \le k~,
	\end{equation*}
	with a parity restriction for each $l(e)$. Define $\widetilde{\Delta}_{\mathcal{D}}(\overline{m})$ as the set of positive integer solutions of above, with correct parity. Note that the $k$-diagram $\mathcal{D}$ should satisfy some compatibility conditions with $\overline{m}$ for $\widetilde{\Delta}_{\mathcal{D}}(\overline{m})$ to be non-empty: if we write the connected components of $\mathcal{D}$ by $\mathcal{D}_1,\cdots,\mathcal{D}_h$, then for each $1 \le j \le h$, the sum of indices $m_i$ should be even, with $i$ ranging over the indices of the circuits traversing $\mathcal{D}_j$. Once these relations are satisfied, we may assume that $\overline{m}$ has even integer components (by producing a simple bijection between the solution sets) for the sake of computing the asymptotics. 
	
	Now given a compatible diagram $\mathcal{D}$, its contribution can be re-written as 
	\begin{equation*}\begin{split}
	&\sum_{\overline{l} \in \widetilde{\Delta}_{\mathcal{D}}(\overline{m}) } (1+ o(1)) \frac{n^{-s/2}}{\prod_{p=1}^k n_p^{-1/2}} \prod_{e \in E} \frac{ \left( n_{p_+(e)} \wedge n_{p_{-}(e)}  \right)^{l(e)/2} }{ \left( n_{p_+(e)} \cdot n_{p_{-}(e)} \right)^{l(e)/4} }\\
	&\quad= (1 + o(1))  \frac{n^{-s/2}}{\prod_{p=1}^k n_p^{-1/2}}  n^{(3s-2k)/6} \int_{\Delta_{\mathcal{D}}(\overline{\alpha})} \exp\left( -\sum_{e \in E} |\tau_{p_+(e)} - \tau_{p_{-}(e)}| \right) d\overline{\omega}\\
	&\quad=(1+o(1)) n^{k/6} I_{\mathcal{D}} (\overline{\alpha},\overline{\tau})~,
	\end{split}\end{equation*}
	with a limiting change of variables $dl = 2n^{1/6} d\omega$ as $n \rightarrow \infty$. Hence, we obtain 
	\begin{equation*}
	\begin{split}
	\frac{1}{n^{k/6}} \widetilde{\mathcal{M}}(\overline{m},\overline{t}) = \sum_{\mathcal{D}} I^\mathcal{D}(\overline{\alpha},\overline{\tau}) + o(1)
	\end{split}
	\end{equation*} where the sum is over $k$-diagrams compatible with the parity of $\overline{m}$. A simple combinatorial calculation then establishes the lemma.

\subsection{Proof of Lemma \ref{lem:MM} from Lemma \ref{lem:MMM}}

	For simplicity, let us concentrate on the case $k = 1$. Extension to the general case is straightforward, modulo some combinatorial manipulations. We set $m_1 = 2m$, and fix $\alpha > 0$ so that $2m\sim \alpha n^{1/6}$ as $n\rightarrow\infty$. As a first step, we use the equality (\ref{eq:viachebinv}) to obtain from Lemma \ref{lem:MMM} that the quantity \begin{equation}
	\begin{split}
	\widetilde{\mathcal{M}}^*(2m,t_1):= \frac{\sqrt{n-t_1-2}}{n!} \tr U_{2m}\left( \frac{X_n}{2\sqrt{n-t_1-2}} \right) 
	\end{split}
	\end{equation} enjoys the bound \begin{equation}\label{eq:Cheby_bound}
	\begin{split}
	\widetilde{\mathcal{M}}^*(2m,t_1)	\le 
	Cm \exp(Cm^{3/2}n^{-1/4})
	\end{split}
\end{equation} (possibly with a larger constant $C>0$), as well as the asymptotics \begin{equation}\label{eq:Cheby_asym}
\begin{split}
\widetilde{\mathcal{M}}^*(2m,t_1) = n^{1/6} \left( 
2\psi(\alpha,\tau_1)+o(1) \right)
\end{split}
\end{equation}
(see for instance \cite[Proof of Theorem I.2.4]{FS}). We now recall the following identities (see e.g. Snyder \cite{Sn})\begin{equation}\label{eq:Snyder}
	\begin{split}
	\lambda^{2r} &= \frac{1}{(2r+1)2^{2r}} \sum_{m=0}^r (2m+1) {{2r+1}\choose{r-m}}U_{2m}(\lambda)~, \\
	\lambda^{2r-1} &= \frac{1}{(2r)2^{2r-1}} \sum_{m=0}^r 2m { 2r\choose r-m  }U_{2m-1}(\lambda)~.
	\end{split}
	\end{equation} We plug in \begin{equation*}
	\begin{split}
	\lambda = \frac{X_n}{2\sqrt{n-t_1-2}}
	\end{split}
	\end{equation*} to the above formula, take the traces of both sides, and appropriately normalize to obtain 
	\begin{equation}\label{eq:SUM}
	\begin{split}
	\mathcal{M}(2r,t_1) =  \sum_{m=0}^r \frac{2m+1}{(2r+1)2^{2r}} {{2r+1}\choose{r-m}} \widetilde{\mathcal{M}}^*(2m,t_1)~.
	\end{split}
	\end{equation} Note that $\widetilde{\mathcal{M}}^*(0,t_1) = (n-t_1-2)^{1/2}/n! \times \tr U_0 \le n^{1/2}$. We then use a simple inequality \begin{equation*}
	\begin{split}
	{{2r+1}\choose{r-m}}  \le C \frac{2^{2r}}{r^{1/2}} e^{-2m^2/r}
	\end{split}
	\end{equation*} together with \eqref{eq:upper_bound} to obtain the desired bound \begin{equation*}
	\begin{split}
	\mathcal{M}(2r,t_1) &\le \frac{C}{r^{3/2}}\left( n^{1/2} + \sum_{m=1}^r m^2 \exp(-2m^2/r + Cm^{3/2}n^{-1/4} )  \right) \\
	&\le \frac{Cn^{1/2}}{r^{3/2}} \exp(C r_p^3/n)~.
	\end{split}
	\end{equation*}
	
	The desired asymptotic expression \eqref{eq:MM_asym}, in this case, takes the form 
	\begin{equation*}
	\begin{split}
	\mathcal{M}(2r,t_1) &= 2\phi(\alpha,\tau_1) + o(1) \\
	&= \frac{1}{\sqrt{\pi}}\left( \frac{\alpha}{2} \right)^{-3/2} + \int_0^\infty \frac{4\xi d\xi}{\sqrt{\pi\alpha}} e^{-\xi^2}\psi(2\sqrt{\alpha}\xi,\tau_1) + o(1)~.
	\end{split}
	\end{equation*} To arrive at this expression, we choose a large $R > 1$ (independent of $n$), and split the sum into four parts, \begin{equation}
	\begin{split}
	\mathcal{M}(2r,t_1) = \mathcal{I}_1 + \mathcal{I}_2 + \mathcal{I}_3 + \mathcal{I}_4~,
	\end{split}
	\end{equation} where $\mathcal{I}_1$ simply denotes the term $m = 0$ and $\mathcal{I}_{i}$ with $i = 2, 3, 4$ denote the sum restricted to the regions $0 < m < R^{-1}n^{1/6}$, $R^{-1}n^{1/6} \le m \le R n^{1/6}$, and $Rn^{1/6}<m \le r$, respectively. 
		
The de Moivre--Laplace approximation yields
	\[ \mathcal{I}_1 = \frac{1}{(2r+1)2^{2r}} \binom{2r+1}{r} \sqrt{n-t_1-2}
	= (1+o(1)) \frac{\sqrt{n}}{{\sqrt{\pi r^3}}} \to \frac{1}{\sqrt{\pi}} \left(\frac\alpha2\right)^{-3/2} \]
as $n \rightarrow \infty$. An additional application of the de Moivre--Laplace 
approximation combined with  \eqref{eq:asym_expression_perm} yields
\begin{equation*}
\mathcal{I}_2 \to \int_{1/R}^R \frac{4\xi}{\sqrt{\pi \alpha}} e^{-\xi^2} \psi(2\sqrt{\alpha}\xi,\tau_1) d\xi~;
	\end{equation*}
the exchange of limits is justified due to (\ref{eq:upper_bound}). 

Now we let $R \to \infty$. For sufficiently large $R$, by substituting in $r \sim 2\alpha n^{1/3}$, we obtain \begin{equation*}
	\begin{split}
	\mathcal{I}_4 \le \sum_{R \le m/n^{1/6}} \frac{C}{n^{1/6}} \left( \frac{m}{n^{1/6}} \right)^2 \exp\left(- c\left( \frac{m}{n^{1/6}} \right)^2  \right) \le C'\int_R^\infty x^2\exp(-x^2)dx
	\end{split}
	\end{equation*} which vanishes in the limit $R \rightarrow \infty$. The term $\mathcal{I}_2$ can be treated analogously. 
This finishes the proof. \qed

\subsection{Proof of Proposition~\ref{prop'} from Lemma~\ref{lem:MM}}\label{s:prop'}

Let us consider the `moreover' part of the case $k=1$. Assume that  $t/2n^{5/6} \rightarrow \tau$ and $r/2n^{1/3} \rightarrow \alpha$ as $n\rightarrow\infty$. Then 
\[\begin{split} \mathcal{M}^\sym(r,t) 
&= \sum_{m=0}^{n_1-1}  \frac{r}{2\sqrt{n_1}} \frac{1}{\sqrt{n_1-m}} (1 - m/n_1)^{\frac{r}{2}}  
\frac{\sqrt{n_1-m}}{n!} \tr \left(\frac{X_{n_1-m}}{2 \sqrt{n_1-m}}\right)^r\\
&= \sum_{m=0}^{n_1-1}  \frac{r}{2\sqrt{n_1}} \frac{1}{\sqrt{n_1-m}} (1 - m/n_1)^{\frac{r}{2}}  
\mathcal{M}(r,m)~.
\end{split}\]
Divide the sum into two parts $I$ and $I\!I$ corresponding to $m \leq n^{1/2}$ and $m > n^{1/2}$, respectively (here $1/2$ is chosen so that $1/3 < 1/2 < 5/6$). Then $I\!I$ 
tends to zero due to the bound (\ref{eq:MM_bound}). On the other hand,
\[ I - (1 + (-1)^r) \phi(\alpha, \tau)\, \frac{r}{2n_1} \, \int_0^\infty \exp \left\{ - \frac{mr}{2n_1} \right\} dm \to 0 \]
due to (\ref{eq:MM_asym}); the exchange of limits is justified due to (\ref{eq:MM_bound}). Hence
\[\mathcal{M}^\sym(r,t)  = I + I\!I =  (1 + (-1)^r) \phi(\alpha, \tau) + o(1)~,\]
as claimed. Extension to the case $k > 1$ is straightforward. \qed

\section{Concluding remarks}

\subsubsection{Functional limit theorem} It is plausible that the
convergence in Theorem~\ref{thm:main} could be upgraded to
the convergence in the space of random continuous functions. 
One possible approach to this problem would be to follow the
arguments in the proof of \cite[Theorem~3]{Sodin2015}.

\subsubsection{The setting of Borodin-Olshanski}\label{sub:bo}
Let us describe the original setting of \cite{BO}.  There, partitions are allowed to both decay and grow in forward time, at random time moments, with time-dependent rates determined by a given curve. Fix a parametrized curve $C = (u(t),v(t))_{t\in I}$ defined on some interval $I$ of $\mathbb{R}$. We assume that $u(t),v(t) > 0$ and $\dot{u}(t) \ge 0, \dot{v}(t) \le 0$ (that is, $C$ is directed southeast), and the parametrization is such that \begin{equation}\label{eq:t_para}
t = \frac{1}{2} (\ln u - \ln v) + \mathrm{const}~.
\end{equation}
Then consider the Poisson process in $\mathbb{R}^2_{>0}$ with constant density 1. Given a point configuration in  $\mathbb{R}^2_{>0}$ generated by the process, the partition $\Lambda(t)$ is defined by applying the Robinson-Schensted algorithm to the points lying inside the rectangle with vertices $(0,0), (0,v(t)), (u(t),0)$, and $(u(t),v(t))$ (see \cite{BO} for details). This defines a random trajectory $\{ \Lambda(t)\}_{t \in I} $, given a point configuration. It can be easily checked that whenever decay (or growth) happens to the trajectory $\Lambda(t)$, the probability of transitioning from a partition to another coincides with that in our discrete-time setting. 

To describe a limit transition, consider a family of curves $C_\theta = (u_\theta(t),v_\theta(t))$, and assume that there is some constant $T \in \mathbb{R}$ such that $u_\theta(T) v_\theta(T) = \theta $ for all $\theta$. That is, the average of the number of boxes that $\Lambda_\theta(T)$ has equals $\theta$. Introduce the variables \begin{equation*}
x^\theta_j(\tau) = \theta^{-1/6} \left( \Lambda(t(\tau))_j - 2\left(u_\theta(t(\tau))v_\theta(t(\tau)) \right)^{1/2} \right)~,\quad t(\tau) = T + \tau \theta^{-1/6}~,
\end{equation*}
and note that this is consistent with time-scaling from \eqref{eq:scaling} as we expect $\approx 2\tau \theta^{5/6}$ transitions to occur during the time interval $[T,T+ \tau \theta^{-1/6}]$, in view of parametrization \eqref{eq:t_para}. Then \cite[Theorem 4.4]{BO} states that the sequence $(x^\theta_1(\tau) \ge x^\theta_2(\tau) \ge\cdots)$ converges to the Airy$_2$ line ensemble, as $\theta \rightarrow \infty$. 

Therefore, our result corresponds to the case when $C_\theta$ is a family of vertical lines. 
Another special case with $C_\theta$ being the lines $\{ u+v =\mathrm{const} \}$ was proved earlier in \cite{PS}.

\subsubsection{A possible extension}\label{sub:conj} The following construction is motivated by \cite{Bor}. Let $\pi \in S_n$ be a random permutation.
To every subset $A \subset \{1, \cdots, n\}$,
associate a partial permutation $\pi_{n,A} = \pi_n|_A$. The Robinson--Schensted 
correspondence takes $\pi_{n,A}$ to a pair of Young tableaux of the same 
shape $\Lambda^n(A)$.

Now we rescale $\Lambda^n(\cdot)$, as follows. For any nice set 
$B \subset \mathbb{R}_+$ (a finite union of bounded intervals), let 
$A_n(B) = 2n^{5/6} B \cap \mathbb{Z}$, and let
\[ x_j^n(A) = n^{-1/6} \left( \Lambda^n_j(A_n(B)) - 2(n-|A_n(B)|)^{1/2}\right)~. \]
Denote by $X^n$ the  stochastic process formed by $x_j^n$. Is it true 
that  that $X^n$ converges to a limiting object, which is a stochastic process 
$X(B) = (x_j(B))_{j \geq 1}$ indexed by nice subsets $B \subset \mathbb{R}$?
  
\bigskip
\paragraph{Acknowledgment} We are grateful to Vadim Gorin, who first encouraged
us to study  \cite{Ok} and made helpful comments at all stages of our work;
to Alexei Borodin, for various remarks and suggestions and for help with references; 
and to Ohad Feldheim, Grigori Olshanski, and Dan Romik for the critique of the preliminary version which led to numerous improvements.

\bibliographystyle{alpha}
\bibliography{random}

\end{document}